\documentclass[twoside]{article}
\usepackage{xcolor}
\usepackage[accepted]{aistats2020}
\usepackage[utf8]{inputenc} 
\usepackage[T1]{fontenc}    
\usepackage{hyperref}       
\usepackage{url}            
\usepackage{booktabs}       
\usepackage{amsfonts}       
\usepackage{nicefrac}       
\usepackage{microtype}      
\usepackage{lipsum}                     
\usepackage{xargs}
\usepackage{xfrac}
\usepackage{amssymb}
\usepackage{amsmath}
\renewcommand{\k}{\mathbf{k}}
\newcommand{\cov}[2]{\text{Cov}\left[#1, #2\right]}

\newcommand{\tr}{\text{Tr}}

\newcommand{\bs}{{\mathcal B}}

\newcommand{\w}{{\bf w}}
\renewcommand{\u}{{\bf u}}
\renewcommand{\d}{{\bf d}}

\newcommand{\x}{{\bf x}}
\renewcommand{\a}{{\bf a}}
\newcommand{\blambda}{\boldsymbol{ \lambda}}
\newcommand{\K}{{\bf K}}
\newcommand{\bsigma}{\boldsymbol{\sigma}}

\renewcommand{\v}{{\bf v}}
\newcommand{\z}{{\bf z}}

\newcommand{\V}{{\cal V}}
\renewcommand{\rm}{{\mathcal R}}

\newcommand{\bigo}{\mathcal{O}}

\newcommand{\A}{{\bf A}}
\newcommand{\B}{{\bf B}}
\newcommand{\M}{{\bf M}}

\newcommand{\U}{{\bf U}}
\renewcommand{\V}{{\bf V}}
\newcommand{\D}{{\bf D}}
\renewcommand{\cov}{{\bf S}}

\newcommand{\ones}{{\bf 1}}
\newcommand{\C}{{\bf C}}

\newcommand{\E}{{\mathbf E}}

\newcommand{\J}{{\mathbf J}}
\newcommand{\N}{{\mathcal N}}
\newcommand{\R}{{\mathbb{R}}}

\newcommand{\I}{{\mathbf I}}

\renewcommand{\P}{{\mathcal P}}
 
\newcommand{\diag}{{\textbf{diag}}}

\usepackage{amsthm}
\newtheorem{theorem}{Theorem}
\newtheorem{lemma}[theorem]{Lemma}
\newtheorem{corollary}[theorem]{Corollary}

\newtheorem{assumption}{Assumption}

\newtheorem{example}[theorem]{Example}

\theoremstyle{remark}

\usepackage{thmtools,thm-restate}

\usepackage{eqparbox}
\newcommand{\comment}[1]{}

\begin{document}

%

%
\twocolumn[

\aistatstitle{Mixing of Stochastic Accelerated Gradient Descent}


\aistatsauthor{  Peiyuan Zhang* \And Hadi Daneshmand* \And  Thomas Hofmann }

\aistatsaddress{ ETH Zurich \And ETH Zurich \And ETH Zurich} 
]

\begin{abstract}
  We study the mixing properties for stochastic accelerated gradient descent (SAGD) on least-squares regression. First, we show that stochastic gradient descent (SGD) and SAGD are simulating the same invariant distribution. Motivated by this, we then establish mixing rate for SAGD-iterates and compare it with those of SGD-iterates. Theoretically, we prove that the chain of SAGD iterates is geometrically ergodic --using a proper choice of parameters and under regularity assumptions on the input distribution. More specifically, we derive an explicit mixing rate depending on the first 4 moments of the data distribution. By means of illustrative examples, we prove that SAGD-iterate chain mixes faster than the chain of iterates obtained by SGD. Furthermore, we highlight applications of the established mixing rate in the convergence analysis of SAGD on realizable objectives. The proposed analysis is based on a novel \textit{non-asymptotic} analysis of products of random matrices. This theoretical result is substantiated and validated by experiments. 
\end{abstract}

\section{Introduction}\label{sec:intro}
Stochastic variants of gradient based optimization methods have become the \textit{de facto} standard optimization technique for large scale learning problems. Trading off statistical and computational aspects, stochastic approximation methods attempt to obtain high statistical accuracy while keeping the computational per-iteration  costs low~\cite{bottou2008tradeoffs}. The vast empirical success of such methods has motivated a growing body of theoretical studies on stochastic approximation in both convex (e.g.\cite{moulines2011non,schmidt2017minimizing,rakhlin2011making}) and non-convex optimization (e.g. \cite{ge2015escaping,daneshmand2018escaping,zhu2018anisotropic}). Most remarkably, the analysis in~\cite{moulines2011non} establishes fast, non-asymptotic convergence rates for Stochastic Gradient Descent (SGD). Despite this growing understanding of SGD, the inner workings of stochastic \textit{accelerated} first-order methods are still not very well-understood.  Inspired by the success of accelerated schemes (such as stochastic momentum and Adam~\cite{kingma2014adam}) in optimization of deep neural networks, interesting recent works are starting to improve the current theoretical understanding of this empirical observation \cite{jain2018accelerating,schmidt2018fast,dieuleveut2017harder}. 

We here contribute to this line of research, starting from the simplest possible setting relevant for machine learning: ordinary least-squares regression. Namely we consider following problem set-up throughout this paper: 
\begin{align}\label{eq:main_leastsquares}
    \begin{split}
       \min_{\w \in \mathbb{R}^d} \Big( f(\w) & = \frac 12 \E_{\z} \left[ f_\z \right] \Big), \\
       \quad f_\z  = \|y-\w^\top \x \|^2 &, \quad \z:= (\x,y) \sim \P, 
    \end{split}
\end{align}
where $\x \in \R^d$ is the input variable and $y \in \R$ is the response variable. We assume that $\x$ is zero-mean with covariance matrix $\mu \I \preceq \cov \preceq L \I$, where $\mu$ is positive.

\paragraph{Mixing time of SGD.}
Stochastic gradient optimizes $f$ through the following iterative scheme: 
\begin{align} \label{eq:sgdupdate}
    \w_{n+1} = \w_n - \gamma \nabla f_{\z_n}(\w_n)
\end{align}
where $\z_n \stackrel{\text{i.i.d.}}{\sim} \P$ and $\gamma$ is a constant stepsize. $\{ \w_k \}_{k=1}^n$ makes a time-homogeneous Markov chain in $\R^d$. Under regularity assumptions on $\cov$, the chain admits a unique invariant distribution denoted by $\pi_\gamma$ \cite{dieuleveut2017bridging}. \cite{dieuleveut2017bridging} proves that this Markov chain enjoys an exponential mixing time with rate $\mu/L$, namely 
\begin{align}
    W_2^2(\nu(\w_n), \pi_\gamma) \leq C (1-c\mu/L)^n
\end{align}
holds, where $W_2$ is Wasserstein-2 distance and $\nu(\w_n)$ denotes the probability measure induced by random variable $\w_n$. 

\paragraph{Stochastic Accelerated Gradient Descent (SAGD).}
Starting with $\w'_{1} = \w'_{0}$, stochastic accelerated gradient descent optimizes $f$ through the following recurrence
\begin{align} \label{eq:our_accelerated_sgd}
    \begin{split}
         \w'_{n+1} = \w'_n & + \beta (\w'_{n}-\w'_{n-1}) \\
         & - \gamma \nabla f_{\z_n}(\w'_{n} + \alpha (\w'_n - \w'_{n-1})),
    \end{split}
\end{align}
where $\z_n \stackrel{\text{i.i.d.}}{\sim} \P$. Although the above sequence is not Markovian, $\{ \u_k := (\w'_{k},\w'_{k-1}) \}_{k=1}^n$ is a Markov chain running on $\R^{2d}$, since
\begin{align} \label{eq:recurrrence_wns}
    \u_{n+1} = \A_n \u_n+ \begin{bmatrix} 
    \gamma y_n \x_n \\ 
    0
    \end{bmatrix} 
\end{align}
where 
\begin{align*}
    \A_n := \begin{bmatrix} 
    (1+\beta) \I- (1+\alpha) \gamma \x_n\x^\top_n & \alpha \gamma \x_n \x^\top_n - \beta \I \\ 
    \I & 0 
    \end{bmatrix}
\end{align*}

\paragraph{A mixing rate for SAGD.}
 The invariant distribution of $\{ \u_k \}_{k=1}^n$ is $(\pi_\gamma, \pi_\gamma)$, where $\pi_\gamma$ is the invariant distribution of the SGD-chain (see Lemma~\ref{lemma:invariance_sagd}). In this regard, it is natural to ask:
\begin{center}
\textit{
    When does $\{ \u_k \}_{k=1}^n$ exhibit better mixing properties compared to $\{ \w_{k} \}_{k=1}^n$? }
\end{center}
Considering that both of chains are simulating the same invariant distribution $\pi_\gamma$, a chain with better mixing properties is more computationally efficient. Notably, better mixing properties often leads to a better convergence rate for the mean and variance of ergodic average. Also, it may lead to better properties for solutions Poisson equation associated the chain. The solution of Poisson equation, in turn, plays an important role in the establishment of Central Limit Theorem for the ergodic average. In this paper, we further highlight a novel application for mixing-analysis: applications in the convergence of SAGD in realizable cases.  Motivated by these applications of mixing properties, this paper aims at characterizing mixing properties of $\{\u_k\}_{k=1}^n$. Specifically, we will prove that there exits a constant $c \in \R_+$ and a $3d \times 3d$-matrix $\C_\theta(\P)$ determined by first 4 moments of $\P$ such that
\begin{align*}
    W_2^2(\nu(\u_n), (\pi_\gamma,\pi_\gamma)) \leq c \| \C_\Theta(\P) \|^n_{\rho_\epsilon}/\epsilon
\end{align*}
holds where $\| \M \|_{\rho_\epsilon}$ denotes pseudospectrum of matrix $\M$, $\P$ belongs to a broad class of distributions, and $\Theta := \{ \gamma, \alpha, \beta \}$. By means of illustrative examples, we show how our results can be employed to derive the accelerated mixing rate $\bigo((1-\sqrt{\mu/L})^n)$. Then, we show that the mixing rate equates the convergence rate of SAGD in realizable cases. The proposed analysis is based on a novel \textit{non-asymptotic} analysis of products of random matrices. Although  the asymptotic analysis of products of i.i.d. random matrices is an old and rich literature~\cite{furstenberg1960products}, non-asymptotic analyses are rare.

\section{Related Works}

Recent results show that -- despite the potential problem of noise instability and error accumulation (see e.g. \cite{devolder2014first}) -- stochastic accelerated methods can indeed be provably faster than SGD in certain settings~\cite{dieuleveut2017harder,jain2018accelerating,schmidt2018fast}. Among these results, \cite{dieuleveut2017harder} and \cite{jain2018accelerating} focus on least-squares regression (the same setting considered in this paper), but the focus of \cite{schmidt2018fast} is on a more general setting of learning halfspaces.  \cite{dieuleveut2017harder} has shown that stochastic Nesterov's acceleration combined with stochastic averaging accelerates the convergence of stochastic gradient descent on quadratic objectives when $\mu=0$. This combination improves the convergence of stochastic gradient descent from $1/n$ to $1/n^2$ for realizable cases, where $\exists \w_*$ such that $f(\w_*)=0$. \cite{schmidt2018fast} shows that accelerated stochastic gradient descent can obtain an accelerated $\mathcal{O}\left(1-\sqrt{\mu/(\rho^2L)}\right)$ rate, if the following strong growth condition over $f_\z$ holds uniformly in $\w$:
\begin{equation} \label{assum:sgc}
    \E_\z  \left[ \| \nabla f_\z(\w) \|^2 \right] \leq \rho \| \nabla f(\w)\|^2.
\end{equation}
 \cite{jain2018accelerating} proves a modified version of SAGD improves the convergence of SGD. Their results rely on a statistical condition number defined as minimum number $\Tilde{\kappa}$ such that  
 \begin{align} \label{eq:statistical}
     \E \left[ \| \x \|_{\cov^{-1}}^2 \x \x^\top  \right]  \preceq \Tilde{\kappa} \cov
 \end{align}
 holds, which allows to prove convergence with rate $\bigo(1-(\sqrt{\mu/(\Tilde{\kappa} L)}))^n$. By proving that $\Tilde{\kappa} \leq L/\mu$, \cite{jain2018accelerating} shows that their method always outperforms SGD. For Gaussian inputs, their method enjoys the the accelerated $\mathcal{O}\left(1-\sqrt{\mu/(d L)}\right)$ rate.
 
 Yet, the goal of this research is different from three valuable piece of works listed above, i.e. \cite{dieuleveut2017harder}, \cite{jain2018accelerating}, and \cite{schmidt2018fast}. Here, we analyze SAGD through the framework of Markov chains. Our goal is extending the established connection between Markov chain and stochastic optimization in \cite{dieuleveut2017bridging}. This paper formulates the connection religiously for SGD. We aim at extending their result to SAGD. 
\section{Preliminaries}
\paragraph{Notations.}
We will repeatedly use eigenvalue decomposition of the covariance matrix $\cov$ as
\begin{align} \label{eq:cov}
    \begin{split}
        \cov = & \U^\top \diag(\bsigma) \U, \quad \bsigma :=[\sigma_1, \dots, \sigma_d] , \\
        & 0 < \sigma_{1}  = \mu \leq \dots  \leq \sigma_d = L.    
    \end{split}
\end{align}
Using $\cov$, we can rewrite the gradient and stochastic gradient of $f$ as follows
\begin{align} \label{eq:grads}
    \nabla f(\w) & = \cov \w - \E \left[ y \x \right],\\ \nabla f_\z(\w) &= \x \x^\top \w - y\x.
\end{align}
Let $[\M]_{ij}$ be the element $(i,j)$ of matrix $\M$. 
  We will repeatedly use the compact notation $\Theta:=\{\alpha,\beta,\gamma\}$ for a set containing the hyperparameters of SAGD. $\|\cdot\|_p$ denotes $p$-norm. For the sake of simplicity we define $\|\cdot\| := \|\cdot\|_2$.

   Our mixing rate is established in terms of Wasserstein-2 metric defined on the set of probability measures on $(\R^{d},\bs(\R^d))$ with bounded second moment, denoted by $\P_2(\R^d)$ \footnote{Notations are borrowed from \cite{dieuleveut2017bridging}}. More precisely, 
  \begin{align} \label{eq:w2_distance}
      W_2^2(\nu,\mu) = \inf_{p \in \Gamma(\nu,\mu)}  \left( \int \| \v - \w \|^2 p(d\v,d\w) \right) 
  \end{align}
  where for all $p \in \Gamma(\nu,\mu)$, $\nu= \int p(.,\w) \mu(d\w)$ and $\mu = \int p(\v,.)\nu(d\v)$
  . Let $\nu(Z)$ be the probability measure induced by the random variable $Z$. Notation $Z \sim \mu$ is equivalent to $\nu(Z) = \mu$. 
  
\paragraph{Pseudospectrum} \label{sec:pseudospectrum}
Let $\sigma(\M)$ be the set of complex eigenvalues of the non-symmetric matrix $\M$; then the spectral radius of $\M$ is $\rho(\M) = \sup\{ |\z|\; | \; \z \in \sigma(\M) \}$. $\epsilon$-pseudospectrum of $\M$, denoted by  $\sigma_{\epsilon}(\M)$, is defined as 
\begin{align*}
    \sigma_{\epsilon}(\M) := \sigma(\M) \cup \{ \z \in C \; | \; \| (\M - \z \I)^{-1} \| \geq 1/\epsilon \}. 
\end{align*}
Pseudospectrum of $\M$ is a genealization of spectral radius for $\epsilon$-pseudospectrum:  
\begin{align*}
    \rho_\epsilon(\M) = \sup\{ |\z| \; | \; \z \in \sigma_\epsilon(\M) \}.
\end{align*}
As the next Lemma states, pseudospectrum bounds the spectral norm of power of non-symmetric matrices. 
\begin{lemma}[Matrix power and Pseudospectrum (Theorem 9.2\cite{jensen09})] \label{lemma:pseudospectrum_power}
The following holds for any $\epsilon$ and all $n$: 
\begin{align*}
    \| \M^n \| \leq \frac{(\rho_{\epsilon}(\M))^{n+1}}{\epsilon}
\end{align*}
\end{lemma}
 Pseudospectrum is mainly developed for perturbation analysis of non-hermitian matrix~\cite{jensen09}. The result of next lemma shows how a pertubation of a matrix reflects in its pseudospectrum.
\begin{lemma}[Robustness of Pseudospectrum (Theorem 5.12.\cite{jensen09})] \label{sec:robust_pseudo}
For all matrices $\M$, the following holds 
\begin{align*}
    \rho_{\epsilon}(\A + \M) \leq \rho_{\epsilon+\|\M\|}(\A)
\end{align*}
\end{lemma}
Next lemma establishes the connection between pseudospectrum and spectral radius of a matrix. 
\begin{lemma}[Bauer--Fike (Theorem 5.11 of \cite{jensen09})] \label{lemma:baur_fike}
Let $\M$ be a diagonalizable square matrix such that $\M = \V \Lambda \V^{-1}$. Then for $\epsilon>0$, the following holds: 
\begin{align}
    \rho_\epsilon(\M) \leq \rho(\M) + \kappa \epsilon
\end{align}
where $\kappa$ is the condition number of $\V$.
\end{lemma}

\paragraph{Input distribution.}
This paper focuses on a structured input distribution. 
\begin{assumption}[Symmetric input] \label{assume:symmetric_inputs}
$\x$ is generated by an orthogonal transformation of a random vector $\v$ whose coordinates have symmetric distribution. More precisely,
\begin{align} \label{eq:symmetric_inputs}
    \x & = \U \v, \quad \E \left[ \v_i^2 \right] = \sigma_i, \\
    \E \left[ \v_i^4 \right] & = k_i, \quad \k := [k_1, \dots, k_d],
\end{align}
where $\U \in \R^{d\times d}$ is an orthogonal matrix  and $\v \in \R^d $ is from a symmetric distribution, i.e. $v_i$ is distributed as $-v_i$\footnote{Exploiting symmetricity of $\v$, one can check that $\U$ is equal to those of Eq.~\eqref{eq:cov}.}.  
\end{assumption}
The above assumption simplifies our theoretical analysis. Notably, all results can be extended to the case where coordinates of $\v$, defined in the last assumption, are independent random variable. We further remark that the above assumption naturally holds in some practical applications, such as speech recognition.

\section{Invariant distribution}
Leveraging the invariance property, next lemma proves that the invariant distribution of SAGD-iterates simulates the same distribution as SGD-iterates.
\begin{lemma} \label{lemma:invariance_sagd}
 If chain $\{ \u_k \}_{k=1}^n$ obtained by the recurrence of Eq.~\eqref{eq:recurrrence_wns} admits a unique invariant distribution, then the invariant distribution is $(\pi_\gamma, \pi_\gamma)$ where $\pi_\gamma$ is the invariant distribution of SGD-iterates. 
\end{lemma}
\begin{proof} 
Proof is based on a simple application of the invarince property. Suppose that $\u_1 =[\w'_1,\w'_0]$ is drawn from the invariant distribution associated with $\{\u_k\}_{k=1}^n$. Since $[\w'_{2},\w'_1]$ is distributed as $[\w'_1, \w'_{0}]$, $\w'_0$ and $\w'_1$ are identically distributed. Since the invariant distribution is assume to be unique, we need to show that $(\pi_\gamma, \pi_\gamma)$, i.e. $\w'_{0,1} \sim \pi_\gamma$, is invariant with respect to the SAGD-update in Eq.~\eqref{eq:our_accelerated_sgd}. The SAGD update for the particular case of ridge-regression can be written alternatively as 
\begin{align*}
    \w'_{2} & = \widehat{\w} + \left( \beta \I - \gamma\x \x^\top\right) \left( \w'_{1} - \w'_{0} \right),
\end{align*}
where $\widehat{\w} = \w'_1 - \gamma \nabla f_{\z} (\w'_1)$. Since $\pi_\gamma$ is the invariant with respect to SGD-update, $\widehat{\w} \sim \pi_\gamma$. It remains to prove that $\w'_2 \sim \pi_\gamma$. 
Recall the definition of $\Gamma(\nu(\w'_2),\nu(\widehat{\w}))$ used in $W_2$ notation at Eq.~\eqref{eq:w2_distance}. Suppose that $\w'_0 = \w'_1 \sim \pi_\gamma$. For this particular case, the joint distribution $(\nu(\w'_2),\nu(\widehat{\w}))$ belongs $\Gamma(\nu(\w'_2),\nu(\widehat{\w}))$, hence $W_2(\nu(\w'_2),\nu(\widehat{\w}))=0$. This concludes the proof: $\w'_2 \sim \pi_\gamma$. 
\end{proof}
\section{Mixing analysis}
In the last section, we prove that SAGD and SGD are simulating the same invariant distribution. Yet, the mixing time for SAGD is unknown to the best of our knowledge. In this section, we prove that SAGD-chain is geometrically ergodic. 
\paragraph{A coupling analysis.} Similar to the analysis of SGD in \cite{dieuleveut2017bridging}, we propose a coupling analysis for SAGD. Consider two sequences $\{\u_{k}^{(0)}\}_{k=1}^n $ and $\{\u_{k}^{(1)}\}_{k=1}^n$ of SAGD-iterates starting from two different initial random vectors $\u^{(0)}_0$ and $\u^{(1)}_{0}$, respectively. These sequences are assumed to be coupled by sharing the same sequence of random variables $\{\z_{k}:=(\x_k,y_k)\}_{k=1}^n$ in the recurrence of Eq.~\eqref{eq:our_accelerated_sgd}. More precisely, these sequences are obtained by following iterative schemes: 
\begin{align*}
    \u_{n+1}^{(i)} = \A_n \u_{n}^{(i)} + \begin{bmatrix} 
    \gamma y_n \x_n \\ 
    0
    \end{bmatrix}, \quad i \in \{1, 2 \}. 
\end{align*}
Next Theorem proves that probability measures $\nu(\u_n^{(0)})$ converges to $\nu(\u^{(1)}_n)$ in an exponential rate. 
\begin{theorem} \label{thm:convergence_char}
Suppose that Assumption~\ref{assume:symmetric_inputs} holds. Let $\{\u_{k}^{(0)}\}_{k=1}^n $ and $\{\u_{k}^{(1)}\}_{k=1}^n$ be two sequences of SAGD-iterates coupled with $\{\z_{k}\}_{k=1}^n$; then, 
\begin{align*}
     W_2^2(\nu(\u_n^{(0)}),\nu(\u_n^{(1)}))\leq  c \| \C_{\Theta}(\P) \|_{\rho_\epsilon}^{n+1}  /\epsilon,
\end{align*}
holds for all $\epsilon>0$ where \footnote{Vectors $\bsigma \in \R^d $, $c:= 18 d^{3/2} \E \| \u_0^{(1)}- \u_0^{(0)}\|^2 $ and $\k \in \R^d $ are defined in Eq.~\eqref{eq:cov} and \eqref{eq:symmetric_inputs}, respectively.} 
\begin{align} 
    \C_{\Theta}(\P) & = \begin{bmatrix} \D_1^2 + (1+\alpha)^2 \K & 2 \D_1 & \I \\ 
     \D_1 \D_2 - \alpha(1+\alpha) \K& \D_2 & 0 \\ 
     \D_2^2 +\alpha^2 \K & 0 & 0
    \end{bmatrix}  \label{eq:cmatrix}
\end{align}
and
\begin{align*}\nonumber
    \D_1 & = (1+\beta) \I - \gamma (1+\alpha) \diag(\bsigma), \\ 
    \D_2 & =  \alpha \gamma \diag(\bsigma) - \beta \I \\
    \K & = \gamma^2 \left( \diag(\k - 2(\bsigma)^2) + \bsigma \bsigma^\top \right). 
\end{align*} 
\end{theorem}
Section~\ref{sec:convergence_char_proof} outlines the proof of the last Theorem. An immediate consequence of the above result is the uniqueness of the invariant distribution when $\| \C_{\Theta}(\P)\|_{\rho_\epsilon}<1$. Later, we will show how we can choose parameters to achieve a fast mixing. 
\paragraph{An exponential rate for Ergodicity}
Replacing $\u^{(0)}_0 \sim (\pi_\gamma, \pi_\gamma)$ into the result of last Theorem leads to a mixing rate for SAGD. Next Corollary states this mixing result.  
\begin{corollary} [Mixing of SAGD] \label{cor:mixing_SAGD}
Suppose that Assumption~\ref{assume:symmetric_inputs} holds and $\{ \u_k \}_{k=1}^n$ are obtained from Eq.~\eqref{eq:our_accelerated_sgd}; then,
\begin{align}
     W_2^2(\nu(\u_k),(\pi_\gamma,\pi_\gamma))\leq  c' \| \C_{\Theta}(\P) \|_{\rho_\epsilon}^{n+1}  /\epsilon
\end{align}
holds for all $\epsilon>0$ as long as $\| \C_{\Theta}(\P) \|_{\rho_\epsilon}<1$, where $\C_\Theta(\P)$ is defined in Eq.~\eqref{eq:cmatrix} and constant \[ c':=18d^{3/2}\E \left[ \| \u_0 - \E_{\u \sim (\pi_\gamma,\pi_\gamma)} \left[ \u \right] \| \right]. \]
\end{corollary}
  As a result, the mixing of SAGD depends on the 4th moment of the input due to the dependency of $\C_\Theta(\P)$ on vector $\k$ that goes into the Matrix $\K$, which in turn arises from the stochastic gradient estimates. This is in contrast to the convergence of (deterministic) accelerated gradient descent which depends only on the covariance matrix of the input.

\section{Spectral analysis and parameter tuning}

 A closer look at Eq.~\eqref{eq:cmatrix} conveys that a small stepsize choice reduces the contribution of $\gamma^2 \K$ in $\C_\Theta(\P)$ which thus reduces the noise effect in the convergence rate. Yet, the optimization process slows down for a small $\gamma$. But how can we find the proper choice of the stepsize to balance this trade-off?
  Given the eigenvalues of the covariance matrix $\cov$ as well as the vector $\k$, one can construct the matrix $\C_{\Theta}(\P)$ and minimize $\| \C_{\Theta} (\P)\|_{\rho_\epsilon}$ in $\Theta$. Since there are only 3-parameters to estimate, this problem can be solved using a simple grid search. We further simply this optimization problem based on the following key observation presented in the next Lemma: $\C_{\Theta}(\P)$ has a diagonal structure that can be employed to bound $\| \C_{\Theta}(\P)\|_{\rho_\epsilon}$ by spectral-radiuses of $3\times 3$-matrices. 
\begin{lemma}[A bound on the mixing rate of SAGD] \label{lemma:spectral_radius_bound}
The spectral radius of matrix $\C_{\Theta}(\P)$ is bounded as 
\begin{align}
    \begin{split}
        \| \C_{\Theta}(\P) \|_{\rho_\epsilon} & \leq \max_{i=1,\dots,d} \| \J_i(\Theta)\|_\rho + \epsilon \\ 
        & + 3(1+\alpha)^2 \gamma^2  \| \diag(\bsigma)^2 - \bsigma \bsigma^\top \| 
    \end{split}
\end{align}
where $\J_j(\Theta)$ is a $3 \times 3$ matrix:
\begin{align*}
        \J_{i}(\Theta) := 
         \begin{bmatrix} 
        \J_i^{(1)}   &  2 [\D_1]_{ii}  & 1 \\
        \J_i^{(2)}  & [\D_{2}]_{ii} & 0 \\ 
        \J_i^{(3)} & 0 & 0
        \end{bmatrix}, 
\end{align*}
and 
\begin{align*}
    \J_i^{(1)} & := [\D_1]_{ii}^2 + (1+\alpha)^2 \gamma^2 (k_{i} - \sigma^2_i) \\ 
    \J_i^{(2)} & := [\D_1]_{ii}[\D_2]_{ii} -  \alpha (1+\alpha)  \gamma^2 (k_i - \sigma_i^2) \\ 
    \J_i^{(3)} & := [\D_{2}]_{ii}^2 + \gamma^2 \alpha^2 (k_i - \sigma_i^2).
\end{align*}
\end{lemma}
For the proof of the last lemma, we refer readers to Appendix \ref{sec:spectral_analysis}.  
Assuming that stepsize $\gamma$ is sufficiently small,  $ \|\J_1(\Theta)\|_{\rho} \geq \|\J_i(\Theta)\|_{\rho} \geq \| \J_d(\Theta) \|_{\rho}$ holds. Hence, the stochastic acceleration ties to 4 parameters: (i) smoothness parameter $\sigma_{d}= L$, (ii) strong convexity $\sigma_{1} = \mu$ and (iii,iv) 4th order statistics $k_d$ and $k_1$. For the choice of $\Theta=\{\alpha,\beta,\gamma\}$, one can solve the following 3 dimensional problem 
\begin{align}
    \begin{split} 
        \min_{\Theta} \quad & \| \J_1(\Theta) \|_\rho \\
        \text{subject to} \quad & \| \J_d(\Theta) \|_\rho \leq 1-c\sqrt{\mu/L}.
    \end{split}
\end{align}
This provides us a practical method for the acceleration of mixing time based on minimal statistics from the input, including $(\sigma_1,\sigma_d)$ and $(k_1,k_d)$. 

\section{Examples}
But, does the result of Theorem~\ref{thm:convergence_char} lead to an accelerated mixing time for SAGD, faster than the mixing rate of SGD? By means of two examples, we show that accelerated mixing rate $\bigo((1-\sqrt{\mu/L})^n)$ is achievable. 

\begin{example}[Gaussian inputs] \label{exam:gaussian_example} $\x$ is a 2-dimensional Gaussian random vector with zero mean, i.e. ${\x \sim \N(0, \diag([\mu,1]))}$. 
\end{example}

By combining Lemma~\ref{lemma:spectral_radius_bound} and Theorem~\ref{thm:convergence_char}, next Lemma established the accelerated rate. 
\begin{lemma}[The acceleration on Example~\ref{exam:gaussian_example}]  \label{lemma:gaussian_gaurantees}
Suppose input and label distributions are those of Example~\ref{exam:gaussian_example}. 
Consider SAGD with parameters: $\alpha = 2$, $\beta=1-10^{-1/2}\sqrt{\mu}$ and $\gamma = 0.1$. Then,
\begin{multline*}
    W_2^2(\nu(\u_n),(\pi_\gamma,\pi_\gamma)) \\ \leq \frac{1200}{\sqrt{\mu}} \left( 1- \sqrt{\mu}/5  \right)^n \E\| \w_0 - \w_* \|^2.
\end{multline*}
holds as long as $\mu\leq0.02$. 
\end{lemma}
We postpone the proof to Section \ref{sec:converge_examples_app} in the appendix.
 The result of the last lemma highlights that the sequence of SAGD-iterates enjoys better mixing properties compared to those of SGD -- if the parameters are chosen properly. Remarkably, our parameter choice implies that more extrapolation (i.e. $\alpha> \beta$) is needed in stochastic settings. In our experiments, we observe that this choice of $\alpha$ is indeed very important for the convergence rate. Therefore, it is very importance to tune parameters using the proposed spectral analysis in the last section. We note that above guarantee readily extends to non-Gaussian data with the same 4-order statistics, since the convergence only depends on the first four moments. Next corollary states this extension. 
\begin{corollary} 
Suppose $\x \in \R^2$ is generated from an orthogonal transformation of a random vector $\z$, i.e. $\x = \U \z$. If the coordinates of $\z$ are drawn from a symmetric zero-kurtosis distribution, then the result of Lemma~\ref{lemma:gaussian_gaurantees} holds (with the same rate using the same parameters). 
\end{corollary}
Note that the zero-kurtosis property in the last corollary guarantees that the first 4 moments of the input distribution match those of a Gaussian distribution. Yet, this condition is not necessary for the accelerated mixing. Next example presents an other input distribution on which SAGD enjoys the accelerated mixing rate $\bigo((1-\sqrt{\mu/L})^n)$ in terms of Wasserstein-2 distance. 
\begin{example}[Uniform-Rademacher input]
\label{exam:uniform_example}
    $\x$ is a two dimensional random variable. The first coordinate of $\x$ is a Rademacher random variable. The second coordinate is uniform on $[- \kappa^{-1/2}, \kappa^{-1/2}]$ for $\kappa<1$. In this case $\mu=1/2$ and $L=\kappa^{-1}/3$ (see Lemma~\ref{lemma:uniform_convergence_app} in Appendix).
\end{example}
The next lemma proposes an accelerated mixing rate for SAGD on the above example.
\begin{lemma}[Results on Example~\ref{exam:uniform_example}]  \label{lemma:uniform_gaurantees}
Suppose the sequence $\{\u_k\}_{k=1}^n$ is obtained by running SAGD on Example~\eqref{exam:uniform_example}. If $\alpha=2$, $\beta=1-10^{-1/2}\sqrt{\kappa}$ and $\gamma=\kappa/10$, then
\begin{align*}
    W_2^2(\nu(\u_n),(\pi_\gamma,\pi_\gamma)) & \leq \frac{1200}{\sqrt{\kappa}} \left( 1- \sqrt{\kappa}/5 \right)^n \| \w_0 - \w_* \|^2
\end{align*}
holds as long as $\kappa = 2\mu/3L \leq0.02$.
\end{lemma}
 Finally, we stress the fact that the above guarantees are only exemplary. As a matter of fact, our approach can be employed for all datasets obeying Assumption~\ref{assume:symmetric_inputs}. 
 \section{Applications in realizable least-squares}
As mentioned in the introduction, mixing properties play roles in the convergence of ergodic average, central limit theorem, and even optimization in over-parameterized settings. In this section, we particularly highlight applications in over-parameterized settings. This setting has attracted attentions due to recent observations in optimization of deep neural networks. Deep nets are almost perfectly optimizable in the sense that simple gradient methods achieve a zero-training error on these networks. This is often attributed to over-parameterized weight-spaces of neural networks that may contain billions parameters. Inspired by this, recent optimization-studies  have focused on the particular case of realizable models, where the minimal objective error zero is achievable ~\cite{schmidt2018fast,pillaud2017exponential}. Next lemma proves that the invariant measure of SAGA is a Dirac measure on the minimizer in realizable cases. 
\begin{lemma} 
Suppose that there exits $\w_* \in \R^d$ such that $f(\w_*)=0$. If the invariant measure of SAGD is unique, then it equates $(\delta(\w_*),\delta(\w_*))$ where $\delta(\w_*)$ is the Dirac measure concentrated on $\w_*$. 
\end{lemma} 
\begin{proof}
Since $f(\w_*)=0$, $f_\z(\w_*)=0$ holds almost surely. Hence $\nabla f_\z(\w_*)=0$, in that $(\delta(\w_*),\delta(\w_*))$ is invariant with respect to SAGD-update in Eq.~\eqref{eq:our_accelerated_sgd}. Since the invariant distribution is assumed to be unique, $(\delta(\w_*),\delta(\w_*))$ is THE invariant distribution.  
\end{proof}
Combining the result of last lemma and Corollary~\ref{thm:convergence_char} yeilds a convergence rate for SAGD in realizable cases: 
\begin{align*}
    W_2^2(\nu(\u_n),(\delta(\w_*),\delta(\w_*))) \leq c \| \C_{\theta}(\P)\|^{n+1}_{\rho_\epsilon}/\epsilon.
\end{align*}
Hence, our established mixing rate equates the convergence rate of SAGD in realizable cases.
More interestingly, the established rate is $\bigo((1-\sqrt{\mu/L})^n)$ for examples~\ref{exam:gaussian_example} and \ref{exam:uniform_example} under the realizability assumption. Let us compare this rate with the existing established convergence rate of \cite{schmidt2018fast} for a modified stochastic accelerated scheme. As mentioned in \textit{Related Works} section, this rate relies on constant $\rho$ in strong growth condition in Eq.~\eqref{assum:sgc}. Lemma~\ref{lemma:example1support} and \ref{lemma:uniform_consequence} prove that $\rho>\bigo(\mu/L)$ for these examples, hence the established convergence of \cite{schmidt2018fast} is not better than $\bigo((1-(\mu/L)^{\sfrac{3}{2}})^n)$ on these examples. This comparison highlights the novelty of the established mixing rate as well as the sharpness of our theoretical guarantees.  
\section{Proof outline for Theorem~\ref{thm:convergence_char}}\label{sec:convergence_char_proof} 

In this section, we outline the proof of Theorem~\ref{thm:convergence_char}. Let $\v_k = \u_k^{(0)}- \u_k^{(1)}$. Eq.~\eqref{eq:recurrrence_wns} yields
\begin{align} \label{eq:vn}
    \v_{n} = \B_n \v_0, \quad \B_n = \A_n \A_{n-1} \dots \A_1. 
\end{align}
According to the definition of $W_2$, 
\begin{align*}
    W_2^2( \nu(\u^{(0)}_n), \nu(\u^{(1)}_n)) \leq \E \left[ \| \v_n \|^2 \right] = \E \left[ \| \B_n \v_0 \|^2 \right].  
\end{align*}
Hence, we need to bound the spectral norm of matrix $\B_n$ to establish the desired convergence in $W_2$. Notably, $\B_n$ is obtained by products of random non-symmetric matrices. 
\paragraph{Asymptotic analyses of products of random matrices.}
We can now leverage results of the well-studied field of products of random matrices (see e.g. \cite{bougerol2012products,furstenberg1960products}), which gives interesting asymptotic characterizations of $\|\B_n\|$. For example, one can show that there is a constant $\lambda_1$ (called Lyapunov exponent) such that 
\begin{align*}
   \lim_{n\to \infty} \frac{1}{n} \log(\|\B_n\|) = \lambda_1
\end{align*}
holds. A straight-forward implication of this is that $\|\B_n\|$ --depending on the sign of $\lambda_1$-- grows or decays in an exponential rate. More interestingly, the asymptotic convergence rate is a constant for all random samples $\B_n$. Yet, this result is asymptotic and it is not easy to estimate the exponent $\lambda_1$. In our setting, we need an implicit formulation (or at least a bound) for $\lambda_1$ in terms of the parameters $\Theta$ such that we can tune $\Theta$ and derive a rate. 
\paragraph{A non-asymptotic analysis.}
Here, we establish a non-asymptotic expression for the Lyapunov exponent $\lambda_1$ which clearly reflects the dependency to parameters $\Theta$. 
The expectation of the squared norm of iterate $\v_n$ (derived in Eq.~\eqref{eq:vn}) reads as 
\begin{align} \label{eq:vn_Mn}
   \E \left[  \|\v_n \|^2 \right] \stackrel{\eqref{eq:vn}}{=} \v_0^\top \M_n \v_0, \quad \M_n := \E \left[ \B_n^\top \B_n \right]. 
\end{align}
To bound $\E \|\v_n\|^2$, we thus need to bound $\|\M_n\|$. Exploiting Assumption~\ref{assume:symmetric_inputs}, one can prove that $\M_n$ obeys an interesting block structure, which is deriven in the next lemma.  
\begin{lemma} \label{lemma:Mn_blockstructure}
Assuming~\ref{assume:symmetric_inputs} holds, $\M_n$ decomposes as 
\begin{align} \label{eq:Mn_decomp}
    \M_{n} = \begin{bmatrix} 
    \U^\top \diag(\blambda^{(1)}_n) \U & \U^\top \diag(\blambda^{(2)}_n) \U \\ 
     \U^\top \diag(\blambda^{(2)}_n) \U & \U^\top \diag(\blambda^{(3)}_n) \U,
    \end{bmatrix}
\end{align}
where the orthogonal matrix $\U$ contains the eigenvectors of the covariance matrix $\cov$.
\end{lemma} 
 Remarkably, $\U$ is independent of the number iterations. Hence, we only need to track diagonal matrices $\blambda_n^{(1-3)}$ in the derived expression for $\M_n$ in Eq.~\eqref{eq:Mn_decomp}. The next lemma establishes a closed-form expression for $\blambda_n^{(1-3)}$.  
\begin{lemma} \label{lemma:blambdan}
Consider $\blambda_{n}^{(1)}$--$\blambda_n^{(3)}$ in Eq.~\eqref{eq:Mn_decomp}.  Then
\begin{align*}
   \a_{n} = \C_{\Theta}^n(\P) \ones, \quad \a_{n} =  \begin{bmatrix}
    \blambda_{n}^{(1)} & \blambda_{n}^{(2)} & \blambda_{n}^{(3)} 
    \end{bmatrix}
\end{align*}
holds (see Eq.~\eqref{eq:cmatrix} for the exact expression of $\C_{\Theta}(\P)$). 
\end{lemma}
For the proof of the last two Lemmas, we refer the reader to Corollary \ref{sec:proof_of_theorem_appendix} in the appendix. Remarkably, the result of the last Lemma allows to compute the eigenvalues of $\M_n= \E\left[ \B_n^\top \B_n \right] $ in a closed form\footnote{Exploiting the block diagonal structure of $\M_n$, we can extract eigenvalues values of $\M_n$ from $\a_n$.}. Furthermore, it obtains a closed form for the Lyapunov exponent: $\lambda_1^2 = \| \a_n \|_{\max}$. Combining the result of the Lemma~\ref{lemma:Mn_blockstructure} and \ref{lemma:blambdan} concludes the proof of Theorem~\ref{thm:convergence_char}: A straight-forward application of the pseudospectrum properties (developed in Lemma~\ref{lemma:pseudospectrum_power}) yields 
\begin{align*} \label{eq:an_bound}
    \| \a_n \| \leq 3d \| \C_\Theta \|^{n+1}_{\rho_\epsilon}/\epsilon. 
\end{align*}
The result of Lemma~\ref{lemma:Mn_blockstructure} concludes the desired bound:  
\begin{align*}
    \E \| \v_n \|^2 \leq \v_0^\top \M_n  \v_0 \leq 18 d^{3/2} \| \C_\Theta \|^{n+1}_{\rho_\epsilon}\E \|\v_0\|^2 /\epsilon. 
\end{align*}

\section{Experiments}
We empirically validate the established result in Theorem~\ref{thm:convergence_char}, and Lemmas~\ref{lemma:gaussian_gaurantees} and~\ref{lemma:uniform_gaurantees}. Then, we empirically show that the result of Theorem~\ref{thm:convergence_char} may hold even on real data sets on which the Assumption~\ref{assume:symmetric_inputs} does not necessarily hold. 
\paragraph{Experiments on examples.}
Through an experiment, we check whether the accelerated rate established in Lemmas~\ref{lemma:gaussian_gaurantees} and~\ref{lemma:uniform_gaurantees} are achievable. Recall that these results hold for examples~\ref{exam:gaussian_example} and \ref{exam:uniform_example}. Our experimental results, presented in Figure~\ref{fig:accelerated_rates_real_world_theory}, confirm that SAGD-chain enjoys the accelerated mixing rate $\bigo((1-\sqrt{\mu/L})^n)$ using the choice of parameters in corresponding lemmas. 
\begin{figure}[h!]
    \centering
    \begin{tabular}{c}
        \includegraphics[width=0.4\textwidth]{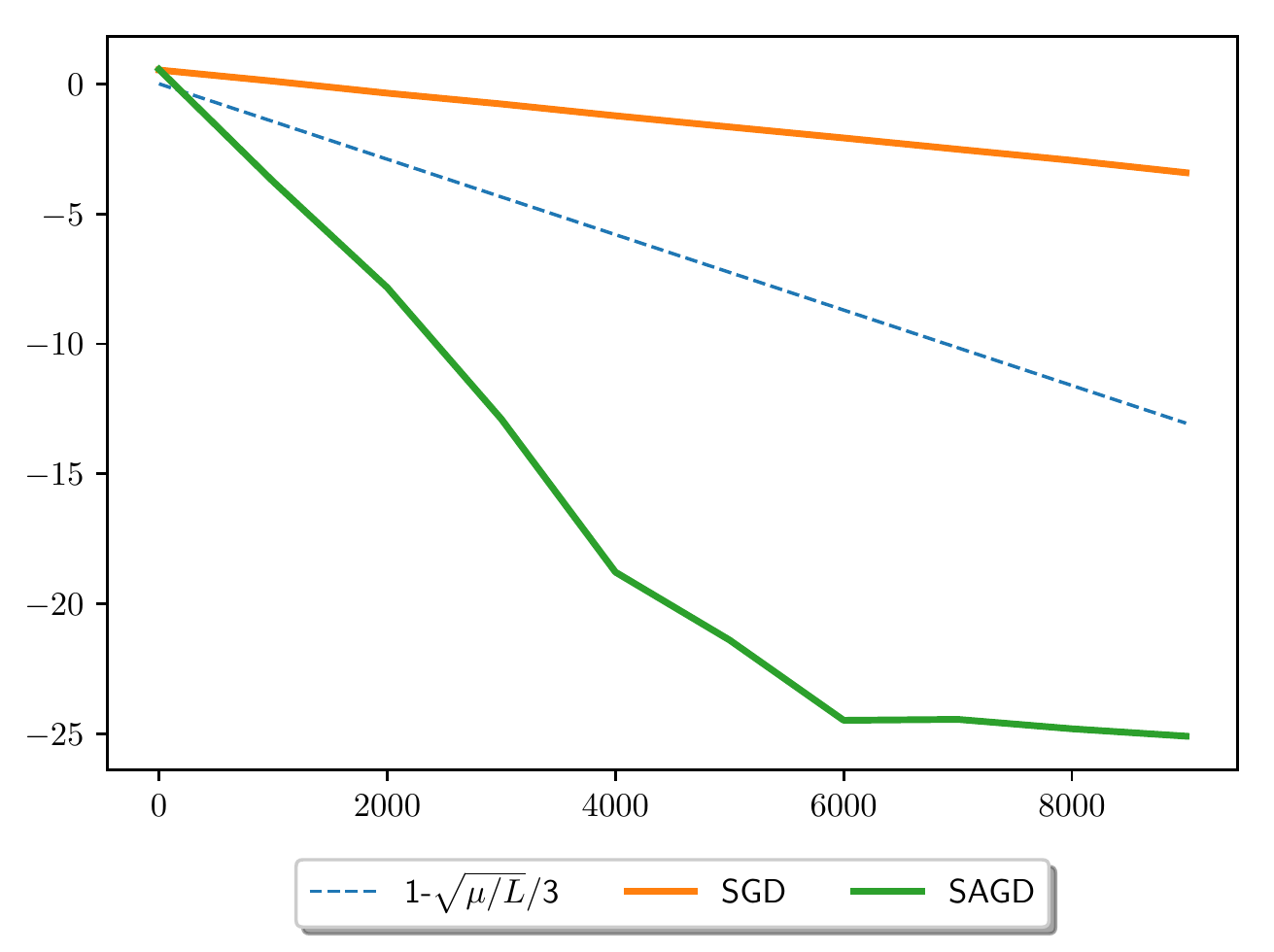}  \\
        Example~\ref{exam:gaussian_example}  \\ \includegraphics[width=0.4\textwidth]{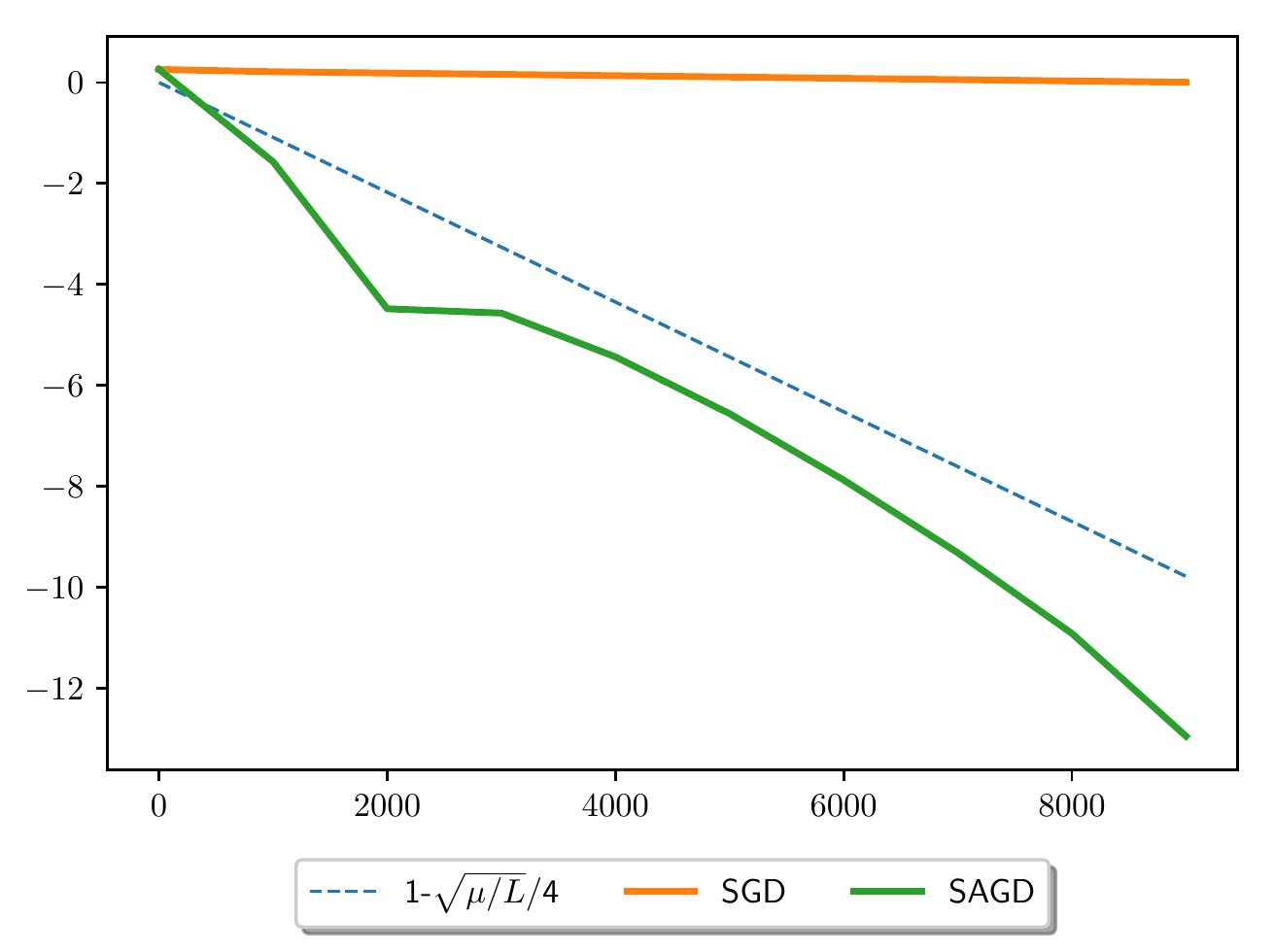}  \\
        Example~\ref{exam:uniform_example} 
    \end{tabular}
    \caption{Mixing of SAGD on examples: horizontal axis shows number of iterations and vertical axis is $\log_{10}\E\| \u_{n}^{(0)} - \u_n^{(1)} \|^2 $ where the expectation is taken over 5 independent runs. Remarkably, $W_2^2(\nu(\u_n^{(0)}),\nu(\u_n^{(1)})) \leq \E\| \u_{n}^{(0)} - \u_n^{(1)} \|^2$. See Lemma~\ref{lemma:gaussian_gaurantees} and ~\ref{lemma:uniform_gaurantees} for details on the parameter choice.  } 
    \label{fig:accelerated_rates_examples}
\end{figure}

\paragraph{Validation of the establish mixing rate.}
Next, we compare the theoretical mixing rate established in Theorem~\ref{thm:convergence_char}, with empirical ones under different parameter configurations $\Theta$. Experiments run on  Example~\ref{exam:gaussian_example} and \ref{exam:uniform_example} and the comparison is represented in Table~\ref{tbl:gaussian} and \ref{tbl:symmetric}, respectively. A total iteration of $n=1000$ is employed for both two examples. This experiment confirms empirical rates are generally consistent with established theoretical mixing rates up to constant factors. We note that we estimate the Wasserstein distance with $\E \left[ \| \u_n^{(0)}- \u_n^{(1)} \|^2 \right]$ for the sake of simplicity. The expectation is taken over 10 independent runs. 

\begin{table}[h]
    \caption{$\E \left[ \| \u_n^{(0)}- \u_n^{(1)} \|^2 \right]$ for Example 8, $\mu=0.05$} \label{tbl:gaussian}
    \begin{center}
        \begin{tabular}{lll}
            \textbf{$\Theta=(\gamma,\beta,\alpha)$}  & Empirical rate &  Theoretical rate \\
            \hline \\
            $(10^{-1},0.95,2)$  & $e^{-0.0605n}$  & $ e^{-0.0568n}$ \\
            $(10^{-1},0.99,2)$  & $e^{-0.0164n}$  & $ e^{-0.0156n}$ \\
            $(10^{-1},0.95,3)$  & $ e^{-0.0628n}$  & $e^{-0.0623n}$ \\
            $(10^{-2},0.95,2)$  & $e^{-0.0260n}$  & $e^{-0.0257n}$ \\
        \end{tabular}
    \end{center}
\end{table}

\begin{table}[h]
    \caption{$\E \left[ \| \u_n^{(0)}- \u_n^{(1)} \|^2 \right]$ for Example 11, $\kappa^{-1}=20$} \label{tbl:symmetric}
    \begin{center}
        \begin{tabular}{lll}
            \textbf{$\Theta=(\gamma,\beta,\alpha)$}  & Empirical rate & Theoretical rate \\
            \hline \\
            $(2\times10^{-3},0.95,2)$  & $e^{-0.0626n}$   & $e^{-0.0527n}$ \\
            $(2\times10^{-3},0.99,2)$  & $e^{-0.0188n}$  & $ e^{-0.0118n}$ \\
            $(2\times10^{-3},0.95,3)$  & $e^{-0.0572n}$  & $ e^{-0.0537n}$ \\
            $(4\times10^{-4},0.95,2)$  & $e^{-0.0206n}$  & $e^{-0.0086n}$ \\
        \end{tabular}
    \end{center}
\end{table}

\paragraph{Experiments on real-world data.}
We substantiate our results on two real-world datasets: Boston Housing and California Housing. These datasets are 13 and 8 dimensional respectively. As a reprocessing step, we normalized inputs (this guarantees $L<1$). We further used $\ell_2$ regularization with the penalty factor $10^{-3}$ (this guarantees $\mu>10^{-3}$). The experiments compare contraction rate, established in Theorem~\ref{thm:convergence_char}, on these datasets. Although Assumption~\ref{assume:symmetric_inputs} does not hold for these datasets, we can reuse the choice of parameters in Lemma~\ref{lemma:gaussian_gaurantees} to achieve an accelerated rate. 

 \begin{figure}[h!]
    \centering
    \begin{tabular}{c c}
        \includegraphics[width=0.4\textwidth]{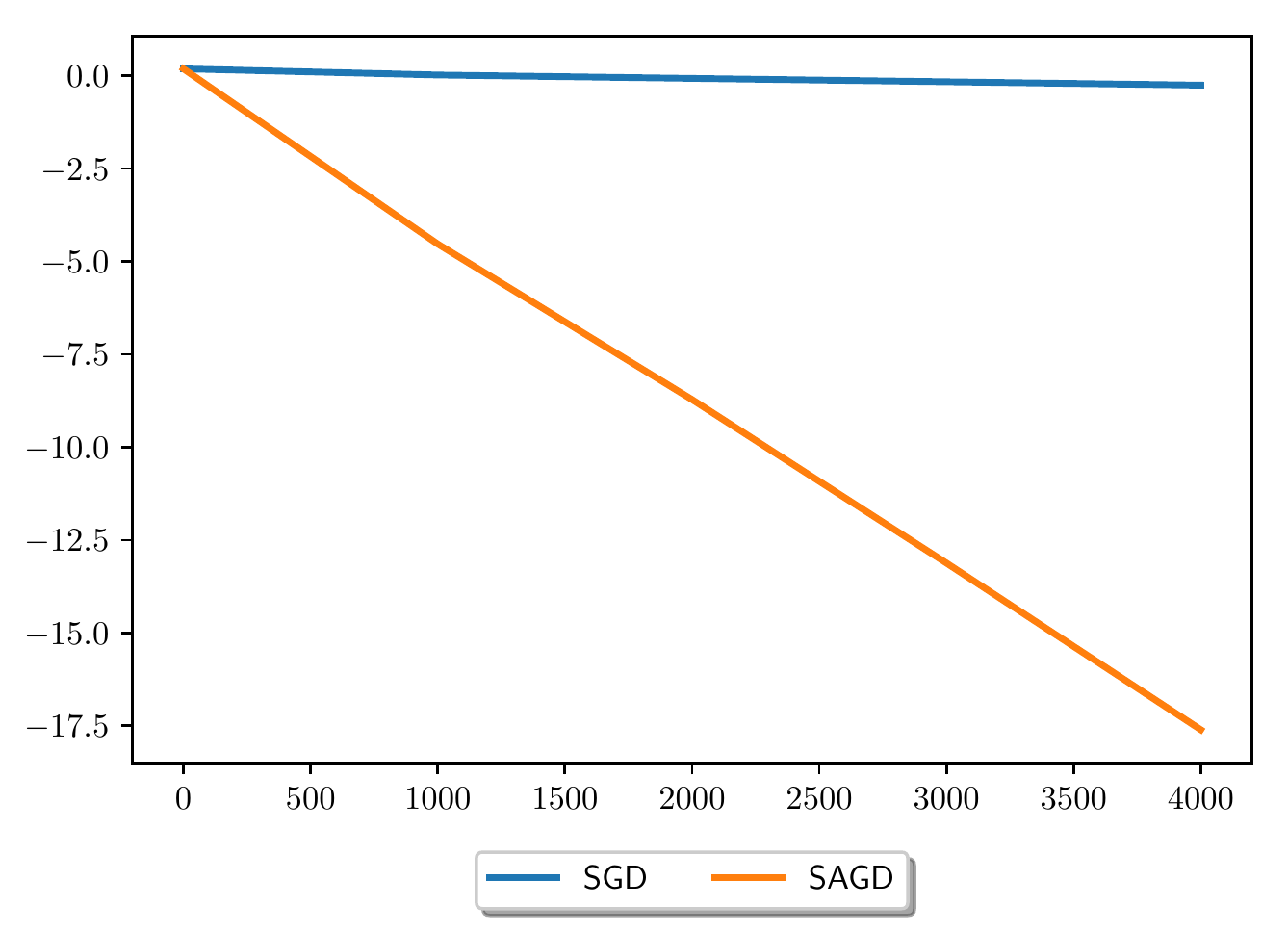}  \\
        Boston Housing  \\ 
        \includegraphics[width=0.4\textwidth]{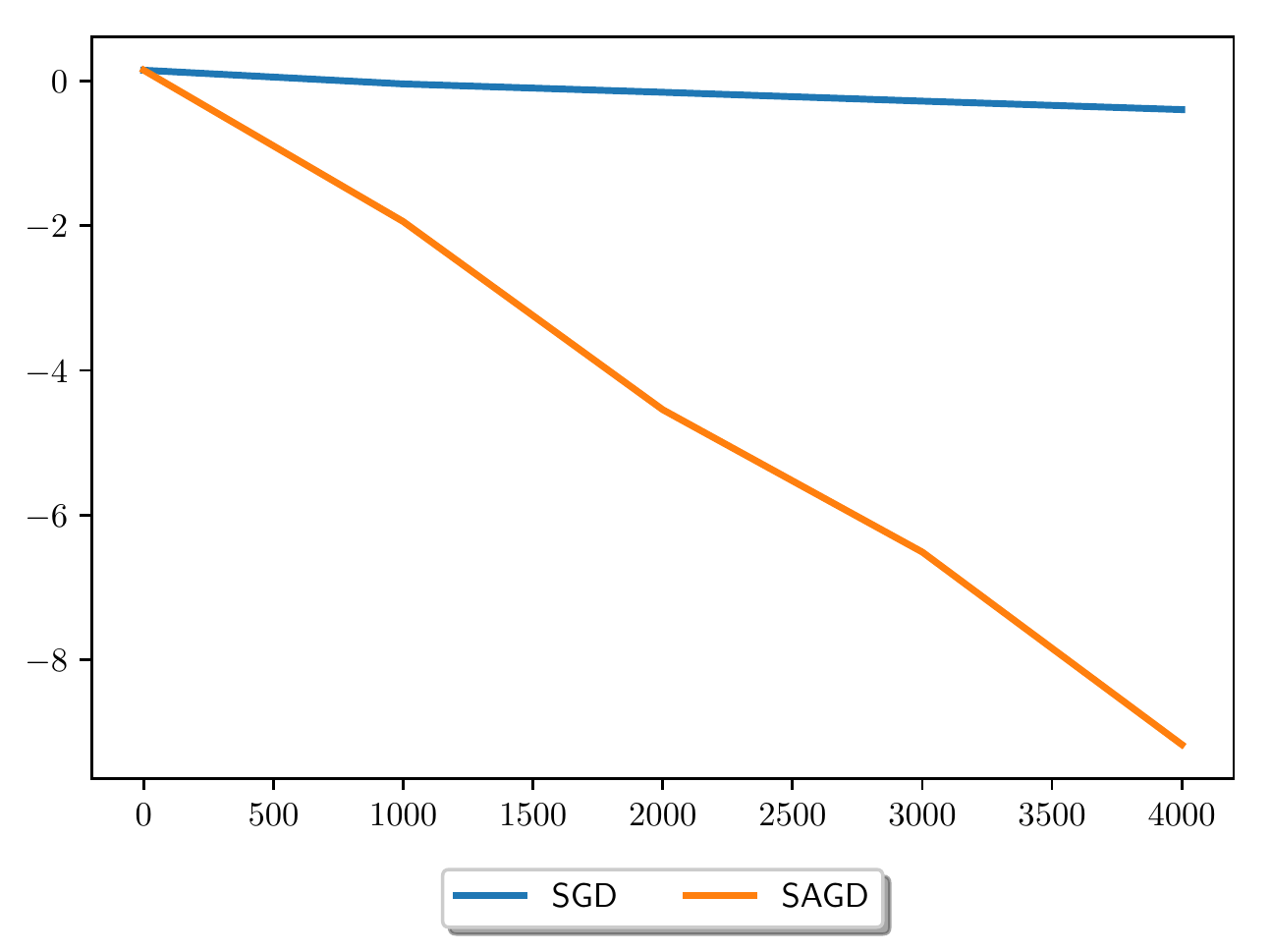}  \\
        California Housing 
    \end{tabular}
    \caption{Accelerated rates achieved on real-world data for theoretical parameters. The horizontal axis shows the number of stochastic iterations $n$ and the vertical axis is $\log_{10}\| \u_{n}^{(0)} - \u_n^{(1)} \|^2 $. }
    \label{fig:accelerated_rates_real_world_theory}
\end{figure}

\section{Discussions} \label{sec:discussions}
We have established the mixing rate for stochastic accelerated gradient descent on least-squares. Using examples, we have shown than that these iterates can mix faster than SGD-iterates depending on the first 4 moments of the input distribution. This result inspires two important follow-up topics: (i) mixing analysis of SAGD on more general optimization problems, and (ii) relaxing the regularity assumption~\ref{assume:symmetric_inputs} on the input distribution.  


\bibliographystyle{plain}
\bibliography{accelerated-mc.bib}
\newpage
\onecolumn
\newpage
\section{Supplementary}

\subsection{Consequences of our Assumptions}  \label{sec:examples}
We repeatedly use the following result on input distributions of assumption~\ref{assume:symmetric_inputs}.
\begin{lemma} \label{lemma:symmetric_inputs}
Suppose matrix $\M$ can be diagonalised as 
\begin{align} 
\M = \U \diag(\blambda) \U^\top
\end{align}
Under Assumption~\ref{assume:symmetric_inputs} holds, the following holds:
\begin{align} 
 \E \left[ \x \x^\top \M \x \x^\top \right] = \U^\top \diag(\blambda') \U, \quad \blambda'_i 
 = \left(\diag\left(\k- (\bsigma)^2\right)  + \bsigma \bsigma^\top\right) \blambda
\end{align}
\end{lemma}
\begin{proof}
Using the structure of $\M$, we first simplify the expression of $\x^\top \M \x$ :  
\begin{align} 
 \x^\top \M \x & = \v^\top \U^\top \U \diag(\blambda) \U^\top \U \v \\ 
 & = \v^\top \diag(\blambda) \v = \sum_{i=1}^d \v_i^2 \blambda_i 
\end{align} 
Replacing this into the desired expectation yields 
\begin{align} \label{eq:symmetric_xMx_expansion}
 \E \left[ \x \x^\top \M \x \x^\top \right]  = \U \E \left[ \left( \sum_{i=1}^d \v_i^2 \blambda_i 
 \right) \v \v^\top \right] \U. 
\end{align} 
Since distribution of $\v_i$ is symmetric, we conclude 
\begin{align} \label{eq:off_diagonal_symmetric}
 \E \left[ \left( \sum_{i=1}^d \v_i^2 \blambda_i 
 \right) \v_p \v_q \right] =  \E \left[ \left( \sum_{i=1}^d \v_i^2 \blambda_i 
 \right) -\v_p \v_q \right] = 0
\end{align} 
holds for all $\v_p \neq \v_q$. For $p=q$, 
\begin{align} 
\E \left[  \left( \sum_{i=1}^d \v_i^2 \blambda_i 
 \right) \v_p^2 \right] = \E \left[ \v_p^4 \right] - (\E \left[\v_p^2\right])^2 + \langle \E\left[ \v^2 \right], \blambda \rangle \E\left[ \v_p^2 \right]
\end{align}
holds. Replacing the above result together with Eq.~\eqref{eq:off_diagonal_symmetric} into Eq.~\eqref{eq:symmetric_xMx_expansion} concludes the proof.  

\end{proof}

A straightforward application of the above result concludes the following corollary on Gaussian inputs.  
\begin{corollary}
\label{cor:gaussianity_consequence}
Suppose that $\x$ is drawn from a zero-mean multivariate normal distribution, i.e. ${\x \sim \N(0,\cov)}$, then the following holds for every matrix $\M$: 
\begin{align} \label{eq:gaussianity_consequence}
    \E \left[ \x \x^\top \M \x \x^\top \right] = 2 \cov \M \cov + \langle \M, \cov \rangle \cov 
\end{align}
where $ \langle \M, \cov \rangle$ is the inner product of vectorized $\M$ and $\cov$.
\end{corollary}
Using the result of last corollary, we can establish a lower-bound on strong grow constant in Eq.~\eqref{assum:sgc}.
\begin{lemma} \label{lemma:example1support}
    Consider the regression objective $f$ on Gaussian input $\x \sim \N(0,\cov)$ and suppose that $y=0$. There exists a vector $\w_0 \in \R^d$ such that
    \begin{align}
        \E_\z \left[\| \nabla f_\z(\w_0) \|^2 \right] \geq L/\mu \| \nabla f(\w_0) \|^2.
    \end{align} 
\end{lemma} 
\begin{proof} 
Let $\w_0$ be the eigenvector of $\cov$ associated with the smallest eigenvalue of $\cov$, i.e. ${\cov \w_0 = \mu \w_0}$. Then we use the result of corollary~\ref{cor:gaussianity_consequence} to compute expected norm of stochastic gradients as 
\begin{align}
    \E \| \nabla f_\z(\w_0) \|^2 & = \w^\top_0 \E \left[ \x \x^\top \x \x^\top \right] \w_0 \\ 
    & \stackrel{\eqref{eq:gaussianity_consequence}}{=} \w_0^\top \cov^2 \w_0 + \tr[\cov]\w_0^\top \cov \w_0 \\ 
    & = \mu^2 + \tr[\cov] \mu \geq \mu (\mu+L)
\end{align}
The norm of gradient evaluated at $\w_0$ is 
\begin{align}
    \|\nabla f(\w) \|^2 = \w^\top_0 \cov^2 \w_0 = \mu^2
\end{align}
To satisfy assumption of Eq.~\eqref{assum:sgc}, we need to choose $\rho$ such that
\begin{align}
    \rho \geq 1 + L/\mu \implies   \mu (\mu+L) \leq \rho \mu^2
\end{align}
\end{proof}
Correspondingly we have a similar lemma for Example~\ref{exam:uniform_example}.
\begin{lemma}\label{lemma:uniform_consequence}
    Consider the regression objective $f$ on the distribution defined in Example~\ref{exam:uniform_example} and suppose $y=0$. There exists a vector $\w_0$ such that 
    \begin{align}
        \E\|\nabla f_\z(\w_0)\|^2 = (L/2\mu+1)\|\nabla f(\w_0)\|^2.
    \end{align}
\end{lemma}
\begin{proof}
The covariance matrix of the variable $\x$ is
\begin{align}
    \cov :=\begin{bmatrix} 1 & 0 \\ 0 & \kappa^{-1}/3 \end{bmatrix}.
\end{align}
Let $\w_0$ be the eigenvector of $\cov$ associated with the smallest eigenvalue of $\cov$, i.e. $\cov\w_0=\w_0$. Similar to Corollary~\ref{cor:gaussianity_consequence} we can calculate
\begin{align}
    \E\big[\x\x^\top\x\x^\top\big]=\begin{bmatrix}\kappa^{-1}/3+1 & 0 \\ 0 & \kappa^{-2}/5\end{bmatrix}.
\end{align}
Then the expected norm of stochastic gradients is
\begin{align}
    \E\|\nabla f_\z(\w_0)\|^2 & = \w_0^\top\E\big[\x\x^\top\x\x^\top\big]\w_0 \\
    & =\w_0^\top\begin{bmatrix}\kappa^{-1}/3+1 & 0 \\ 0 & \kappa^{-2}/5\end{bmatrix}\w_0 \\
    & = \kappa^{-1}/3+1.
\end{align}
Writing in term of $\mu$ and $L$ leads to conclusion
\begin{align}
    \E\|\nabla f_\z(\w_0)\|^2 = (L/2\mu+1)\|\nabla f(\w_0)\|^2.
\end{align}
\end{proof}

\subsection{Proof of Theorem~\ref{thm:convergence_char}} \label{sec:proof_of_theorem_appendix}
\paragraph{Deterministic--stochastic decomposition.}
Recall matrix $\A_n$ in recurrence of Eq.~\eqref{eq:recurrrence_wns}:
\begin{align}
    \A_n := \begin{bmatrix} 
    (1+\beta) \I- (1+\alpha) \gamma \x_n\x^\top_n & \alpha \gamma \x_n \x^\top_n - \beta \I \\ 
    \I & 0 
    \end{bmatrix}.
\end{align}
$\A_{n+1}$ decomposes as 
 \begin{align} \label{eq:An_decomposition}
     \A_{n+1} = \A -\gamma 
     \epsilon, \quad \epsilon := \begin{bmatrix} 
     (1+\alpha) \Delta & -\alpha\Delta \\ 
     0 & 0 
     \end{bmatrix}, \Delta :=(\cov - \x \x^\top)
 \end{align}
 where $\E \left[ \epsilon \right] = 0$ and matrix $\A$ is 
 \begin{align} \label{eq:A}
     \A = \begin{bmatrix} 
     \A_1 & \A_2 \\ 
     \I & 0 
     \end{bmatrix}, \A_1 := (1+\beta)\I - \gamma (1+\alpha) \cov, \A_2 := \alpha \gamma \cov - \beta \I  
 \end{align}
 The above decomposition allows us to analyze the covariance term induced by the noise separately. Recall that matrices $\M_n$ and $\B_n$ were defined as 
 \begin{align} 
 \M_n = \E \left[ \B_n^\top \B_n \right], \quad \B_n = \A_n \A_{n-1} \dots \A_1  
 \end{align} 
 
 \begin{lemma} 
Suppose assumption~\ref{assume:symmetric_inputs} holds on $\x$. We further assume that $\M$ has the following block structure: 
\begin{align}
      \M_n = \begin{bmatrix} 
      \M_1 & \M_2 \\ 
      \M_2 & \M_3 
     \end{bmatrix}, \quad \M_{i} = \U \diag(\blambda_n^{(i)}) \U^\top
\end{align}

Then matrix $\M_{n+1}$ can be decomposed as
\begin{align} \label{eq:E_nplus_expansion_result}
     \M_{n+1} = \A^\top  \M_n \A + \gamma^2 \underbrace{\begin{bmatrix} (1+\alpha)^2 \K &  -\alpha(1+\alpha) \K \\
 -\alpha (1+\alpha) \K  & \alpha^2 \K 
 \end{bmatrix}}_{\text{due to the noise}}
\end{align}
where  \begin{align}
    \K := \U \diag\left( \left(\diag(\k - 2(\bsigma)^2) + \bsigma \bsigma^\top \right)\blambda_n^{(1)} \right) \U^\top
\end{align}
\end{lemma} 
\begin{proof}
 Using the decomposition of matrix $\A_n$ in Eq.~\eqref{eq:An_decomposition}, we expand the expectation as  
 \begin{align}\label{eq:M_nplus}
     \M_{n+1}  & = \E \left[\A_{n+1}^\top  \M_n \A_{n+1}  \right] \\   
     & = \A^\top  \M_n \A + \gamma^2  \E \left[ \epsilon^\top  \M_n \epsilon \right] 
 \end{align}

We compute 2nd term (induced by the noise) 
\begin{align} \label{eq:variance_noise}
    \E \left[ \epsilon^\top  \M_n \epsilon \right]  = \begin{bmatrix} 
    (1+\alpha)^2 \K & -\alpha (1+\alpha) \K \\ 
    -\alpha (1+\alpha) \K  & \alpha^2 \K
    \end{bmatrix} 
\end{align}
The result of Lemma~\ref{lemma:symmetric_inputs} on symmetric inputs concludes the expression of $\K$. 
\end{proof}

\paragraph{The block structure of matrix $\M_n$.}
Lemmas~\ref{lemma:Mn_blockstructure} and \ref{lemma:blambdan} are consequences of the last lemma.
\begin{corollary} [Combined lemmas \ref{lemma:Mn_blockstructure} and \ref{lemma:blambdan}]
Suppose assumption~\ref{assume:symmetric_inputs} holds on $\x$; then 
\begin{align} \label{eq:Mn_condition}
      \M_n  = \begin{bmatrix} 
      \M_1 & \M_2 \\ 
      \M_2 & \M_3 
     \end{bmatrix}, \quad \M_{i} = \U \diag(\blambda_n^{(i)}) \U^\top
\end{align}
where 
\begin{align} \label{eq:an_appendix}
    \a_n=
     \C^n_\Theta
    \ones, \quad \a_n := \begin{bmatrix} 
    \blambda_n^{(1)} 
    & 
    \blambda_n^{(2)}
    &
    \blambda_n^{(3)}
     \end{bmatrix} 
\end{align}
and matrix $\C_\Theta$ is 
\begin{align}
    \C_\Theta & = \begin{bmatrix} \D_1^2 + (1+\alpha)^2 \K' & 2 \D_1 & \I \\ 
     \D_1 \D_2 - \alpha(1+\alpha) \K'& \D_2 & 0 \\ 
     \D_2^2 +\alpha^2 \K' & 0 & 0
    \end{bmatrix}  \label{eq:C_matrix}\\ 
    \D_1 &= (1+\beta) \I - \gamma (1+\alpha) \diag(\bsigma),\\ 
    \D_2 &=  \alpha \gamma \diag(\bsigma) - \beta \I \\ 
    \K' &= \gamma^2 \left( \diag(\k - 2(\bsigma)^2) + \bsigma \bsigma^\top \right) \label{eq:cmat_appendix}
\end{align}
\end{corollary}
\begin{proof}
  We prove the result by induction. 
  Recall matrix $\A$ with the following block structure: 
\begin{align}
     \A = \begin{bmatrix} 
     \A_1 & \A_2 \\ 
     \I & 0 
     \end{bmatrix}, \quad  \A_1 := (1+\beta)\I - \gamma (1+\alpha) \cov, \quad  \A_2 := \alpha \gamma \cov - \beta \I
 \end{align}
Given the decomposition $\cov = \U \diag(\bsigma) \U^\top$ (in Eq.~\eqref{eq:cov}), submatrices $\A_i$ decomposes as:
\begin{align}
    \A_i = \U \D_i \U^\top 
\end{align}
where $\U$ is eigenvectors of the covariance  and $\D_i$ are diagonal matrices of eigenvalues in Eq.~\eqref{eq:cmat_appendix}.  Using the assumption of  Eq.~\eqref{eq:Mn_condition},  
  we expand the first term in Eq.~\eqref{eq:E_nplus_expansion_result}
\begin{align}
    \A^\top  \M_n \A & = \begin{bmatrix} 
    \A_1 \M_1 \A_1 + 2\M_2 \A_1 + \M_3 & \A_1 \M_1 \A_2 + \M_2 \A_2 \\ 
    \A_2 \M_1 \A_1 + \M_2 \A_2 & \A_2 \M_1 \A_2 
    \end{bmatrix} \\ 
    & = 
    \begin{bmatrix} 
     \B_1 & \B_2 \\ 
     \B_2 & \B_3 
    \end{bmatrix} 
\end{align}
where 
\begin{align}
    \B_1 &= \U \left( \D_1^2 \diag(\blambda_n^{(1)}) + 2 \D_1\diag(\blambda_n^{(2)}) + \diag(\blambda_n^{(3)})\right) \U^\top \\ 
    \B_2 &= \U \left( \D_1 \D_2 \diag(\blambda_n^{(1)}) + \D_2 \diag(\blambda_n^{(2)}) \right) \U^\top \\ 
    \B_3 & = \U \left(  \D_2^2 \diag(\blambda_n^{(1)}) \right) \U^\top 
\end{align}
Using eigendecomposition of $\cov$ and $\M_1$, the matrix $\K$ in the noise-induced term can be written as 
\begin{align}
     \K  = \U \left(  \left( \diag(\k - 2(\bsigma)^2) + \bsigma \bsigma^\top \right) \diag(\blambda_n^{(1)}) \right) \U^\top
\end{align}
Putting all together, we complete the proof: 
\begin{align}
     \M_{n+1}  = \begin{bmatrix} 
    \M_{1}^+ & \M_2^+ \\ 
    \M_2^+ &  \M_3^+
    \end{bmatrix}, \quad  \M_i^+ = \U \diag(\blambda_{n+1}^{(i)}) \U^\top 
\end{align}
where 
\begin{align}
    \a_{n+1} = \C_{\Theta}\a_n
\end{align}
where matrix $\C_{\Theta}$ is those of Eq.~\eqref{eq:C_matrix}. 
The above results concludes our induction. \end{proof}
\paragraph{Bound spectral norm of $\M_n$.}
Given bounds on $\a_n$, how we can establish lowerbound and upperbound on eigenvalues of $\M_n$? Next lemma address this result. The result will be used in the proof of Theorem~\ref{thm:convergence_char} (very last step).  
\begin{lemma} \label{lemma:v1_times_Mn_bound}
Recall matrix $\M_n$ in the last corollary:
\begin{align}
      \M_n = \begin{bmatrix} 
      \M_1 & \M_2 \\ 
      \M_2 & \M_3 
     \end{bmatrix}, \quad \M_{i} = \U \diag(\blambda_n^{(i)}) \U^\top, \quad \a_n = \begin{bmatrix} 
    \blambda_n^{(1)} & 
    \blambda_n^{(2)} &
    \blambda_n^{(3)}
    \end{bmatrix} 
\end{align}
For $\M_n$, the following holds 
\begin{align}
   \| \a_n \|_\infty \| \v_1 \|^2 \leq  \begin{bmatrix} 
     \v_1^\top & \v_1^\top 
    \end{bmatrix} \M_n \begin{bmatrix} 
    \v_1 \\ 
    \v_1 
    \end{bmatrix} 
    \leq  6 \sqrt{d} \|\a_n\| \| \v_1 \|^2.
\end{align}
\end{lemma}
\begin{proof}
The proof is straight-forward 
 \begin{multline} \label{eq:v1_Mn_v1_expansion}
      \begin{bmatrix} 
    \v_1^\top & \v_1^\top 
    \end{bmatrix} \begin{bmatrix} 
    \U \diag(\blambda_n^{(1)}) \U^\top & \U \diag(\blambda_n^{(2)}) \U^\top \\ 
    \U \diag(\blambda_n^{(2)}) \U^\top & \U \diag(\blambda_n^{(3)}) \U^\top
    \end{bmatrix}
    \begin{bmatrix} 
    \v_1 \\ 
    \v_1 
    \end{bmatrix} \\ 
    = \v_1^\top \U^\top \diag\left( \blambda_n^{(1)} + 2 \blambda_n^{(2)} + \blambda_n^{(3)} \right) \U \v_1   
 \end{multline}
Hence 
\begin{align}
    \begin{bmatrix} 
     \v_1^\top & \v_1^\top 
    \end{bmatrix} \M_n \begin{bmatrix} 
    \v_1 \\ 
    \v_1 
    \end{bmatrix} & \leq 2 \left( \sum_{i=1}^3 \|\blambda_n^{(i)}\|_1\right)  \| \U \v_1 \|^2 \\
    & \leq 2 \sqrt{3} \left( \sum_{i=1}^d \| \blambda_n^{(i)}\| \right) \| \v_1 \|^2  \\ 
    & \leq 6 \sqrt{d} \|\a_n\|_2 \| \v_1 \|^2
\end{align}
which concludes the desired upper-bound. For the lower-bound, we use the fact that all $\blambda_n^{(i)}$ have positive coordinates (as they are eigenvalues of symmetric matrices). Plugging this into the Eq.~\eqref{eq:v1_Mn_v1_expansion} together with orthogonality of $\U$ concludes the proof:  
\begin{align}
    \| \a_n \|_\infty \| \v_1 \|^2 \leq  \begin{bmatrix} 
     \v_1 & \v_1^\top
    \end{bmatrix} \M_n \begin{bmatrix} \v_1^\top \\ \v_1 \end{bmatrix}
\end{align}

\end{proof}
\subsection{The spectral analysis} \label{sec:spectral_analysis}
\begin{lemma}[Restated Lemma~\ref{lemma:spectral_radius_bound}] 
The spectral radius of matrix $\C_{\gamma, \alpha,\beta}$ is bounded as 
\begin{align}
    \| \C_{\gamma,\alpha,\beta} \|_{\rho} \leq \max_{i=1,\dots,d} \| \J_i \|_\rho + \epsilon + 3(1+\alpha)^2 \gamma^2   \| \diag(\bsigma)^2 - \bsigma \bsigma^\top \| 
\end{align}
where $\J_j$ is a $3 \times 3$ matrix as
\begin{align}
    \J_{i} := \begin{bmatrix} 
      [\D_1]_{ii}^2 + (1+\alpha)^2 \gamma^2 (k_{i} - \sigma^2_i)   &  2[\D_1]_{ii}  & 1 \\
    [\D_1]_{ii}[\D_2]_{ii} -  \alpha (1+\alpha)  \gamma^2 (k_i - \sigma_i^2)  & [\D_{2}]_{ii} & 0 \\ 
    [\D_{2}]_{ii}^2 + \gamma^2 \alpha^2 (k_i - \sigma_i^2) & 0 & 0
    \end{bmatrix} 
\end{align}
\end{lemma}
\begin{proof}
Recall matrix $\C := \C_{\gamma,\alpha,\beta}$ in Eq.~\eqref{eq:cmatrix} formulation:

\begin{align}
    \C : = \begin{bmatrix} \C_1 & 2 \D_1 & \I \\ 
     \C_2 & \D_2 & 0 \\ 
     \C_3 & 0 & 0
    \end{bmatrix}
\end{align}
where 
\begin{align}
    \C_1 & = \D_1^2 + (1+\alpha)^2 \K', \quad \D_1 := (1+\beta) \I - \gamma (1+\alpha) \diag(\bsigma)\\ 
    \C_2 & = \D_1 \D_2 - \alpha(1+\alpha) \K', \quad \D_2 =  \alpha \gamma \diag(\bsigma) - \beta  \K'\\ \C_3 & = \D_2^2 +\alpha^2 \K', \quad \K' = \gamma^2 \left( \diag(\k - 2(\bsigma)^2) + \bsigma \bsigma^\top \right) 
\end{align}
\paragraph{A perturbation form.}
We can decompose $\K'$ (the induced matrix by noise) to sum of diagonal matrices and a non-diagonal matrix: 
\begin{align}
    \K' := \gamma^2 \left( \diag(\k- \bsigma^2) + \xi \right), \quad \xi = \bsigma \bsigma^\top  - \diag(\bsigma)^2
\end{align}
Using the decomposition of $\K'$, we decompose $\C$ as a perturbation of a block diagonal matrix: 
\begin{align}
    \C = \bar{\C} + \gamma^2 \boldsymbol{\xi}, \quad \boldsymbol{\xi} := \begin{bmatrix} 
    (1+\alpha)^2 \xi & 0 & 0 \\ 
    -\alpha(1+\alpha) \xi & 0 & 0 \\ 
    \alpha^2 \xi & 0 & 0
    \end{bmatrix}
\end{align}
where matrix $\bar{\C}$ is 
\begin{align}
    \bar{\C} := \begin{bmatrix} \D_1^2 + (1+\alpha)^2 \bar{\K}   & 2 \D_1 & \I \\ 
     \D_1 \D_2 - \alpha(1+\alpha) \bar{\K}& \D_2 & 0 \\ 
     \D_2^2 +\alpha^2 \bar{\K} & 0 & 0
    \end{bmatrix}, \quad \bar{\K} :=  \gamma^2 \left( \diag(\k)-\diag(\bsigma)^2  \right)
\end{align}
\paragraph{Exploiting the diagonal structure.} Matrix $\bar{\C}$ has a particular structure: its $d\times d$ blocks are diagonal. Similar to the classical analysis of Heavy ball method \cite{stephen2017behavior}, we permute rows and columns to compute eigenvalues of $\bar{\C}$. Let $\Pi$ be a permutation matrix that swaps column(and row) $i+d$ and $i+2d$ with columns $i+1$ and $i+2$, respectively. Then, 
\begin{align}
    \Pi^\top \bar{\C} \Pi  = \begin{bmatrix} 
     \J_{1} & 0&  0 & \dots & 0 \\ 
     0 & \J_{2} & 0 & \dots & 0 \\ 
     \vdots & \vdots & \vdots & \vdots & \vdots \\ 
     0 & \dots & \dots & 0 & \J_{d} 
    \end{bmatrix} 
\end{align}
where where $\J_j$ is a $3 \times 3$ matrix as
\begin{align}
    \J_{i} := \begin{bmatrix} 
      [\D_1]_{ii}^2 + (1+\alpha)^2 \gamma^2 (k_{i} - \sigma^2_i)   &  2 [\D_1]_{ii}  & 1 \\
    [\D_1]_{ii}[\D_2]_{ii} -  \alpha (1+\alpha)  \gamma^2 (k_i - \sigma_i^2)  & [\D_{2}]_{ii} & 0 \\ 
    [\D_{2}]_{ii}^2 + \gamma^2 \alpha^2 (k_i - \sigma_i^2) & 0 & 0
    \end{bmatrix} 
\end{align}
Given spectral decomposition of $\J_i = \U_i \diag(\d_i) \U_i^\top$, we decompose $\bar{\C}$:
\begin{align}
   \U_J^\top \Pi^\top  \bar{\C} \Pi \U_J = \begin{bmatrix} 
   \d_1 & 0 & \dots & 0 \\ 
    0 & \d_2 & \dots & 0 \\ 
    \vdots & \vdots & \vdots & \vdots \\ 
    0 & 0 & \dots & \d_d
   \end{bmatrix}
   , \quad \U_J :=  \begin{bmatrix} 
    \U_1 & 0 & \dots & 0 \\ 
    0 & \U_2 & \dots & 0 \\ 
    \vdots & \vdots & \vdots & \vdots \\ 
    0 & 0 & \dots & \U_d
    \end{bmatrix} 
\end{align}
Since matrices $\U_j$ and perturbation matrix $\Pi$ are orthogonal, we conclude that eigenvalues of $\bar{\C}$ are those of $\J_i$s (i.e. $\d_i$s). Hence 
\begin{align} \label{eq:Cbar_bound}
    \| \bar{\C} \|_\rho \leq \max_{i=1,\dots, d} \| \J_i \|_\rho 
\end{align}
\paragraph{A bound on pseudospectrum.}
So far, we have established a bound on spectral radius of $\bar{\C}$. Yet, we need to bound the spectral radius of $\bar{\C} + \gamma^2 \boldsymbol{\xi}$. To this end, we use results of results Pseudospectrum in section \ref{sec:pseudospectrum}.   
\begin{align}
    \| \C \|_{\rho_\epsilon} & = \| \bar{\C} + \gamma^2 \boldsymbol{\xi} \|_{\rho_\epsilon} \\
    & \stackrel{\text{Lemma}~\ref{sec:robust_pseudo}}{\leq } \| \bar{\C} \|_{\rho_{\epsilon + \gamma^2 \| \boldsymbol{\xi}\|}} \\
    & \stackrel{\text{Lemma}~\ref{lemma:baur_fike}}{\leq}  \| \bar{\C} \|_\rho  + (\epsilon + \gamma^2 \| \boldsymbol{\xi} \|)
\end{align}
Note that in the last step, we have used the fact that the conditioning of $\Pi \U_J$ is bounded by 1 (which is a consequence of orthogonality of matrices $\Pi$ and $\U_J$). Replacing the above result into Eq.~\eqref{eq:Cbar_bound} concludes the proof: 
\begin{align}
    \| \C \|_{\rho_\epsilon} \leq  \max_{i=1,\dots, d} \| \J_i \|_\rho + \epsilon + \gamma^2 \| \boldsymbol{\xi} \| \leq  \max_{i=1,\dots, d} \| \J_i \|_\rho + \epsilon + 3(1+\alpha)^2 \gamma^2 \| \xi \| 
\end{align}
\end{proof}

\subsection{Examples of improved rates} \label{sec:converge_examples_app}
\begin{lemma}[Restated Lemma \ref{lemma:gaussian_gaurantees}]
Suppose input and label distributions are those of example~\ref{exam:gaussian_example}. 
For $\mu\leq0.02$, consider stochastic accelerated gradient descent with parameters: $\alpha = 2$, $\beta=1-10^{-1/2}\sqrt{\mu}$ and $\gamma = 0.1$. Then,
\begin{align}
    \E \| \w_n - \w_* \|^2 \leq \frac{1200}{\sqrt{\mu}} \left( 1- \sqrt{\mu}/5  \right)^n \| \w_0 - \w_* \|^2.
\end{align}
\end{lemma}

\begin{proof}
Recall $\x\sim\mathcal{N}(0,\diag([\mu,1]))$, as defined in Example \ref{exam:gaussian_example}, and $\bsigma=[\sigma_1,\sigma_2]=[\mu,1]$. Therefore their corresponding fourth moments are
\begin{align}
    k_1=3\mu^2,\quad k_2=3.
\end{align}
Plug $k_i$ and $\sigma_i$ into $\J_i$ 
\begin{align}
    \J_{i} = \begin{bmatrix} 
      [\D_1]_{ii}^2 + 2(1+\alpha)^2 \gamma^2 \sigma^2_i   &  2 [\D_1]_{ii}  & 1 \\
    [\D_1]_{ii}[\D_2]_{ii} -  2\alpha (1+\alpha)  \gamma^2 \sigma_i^2  & [\D_{2}]_{ii} & 0 \\ 
    [\D_{2}]_{ii}^2 + 2\gamma^2 \alpha^2 \sigma_i^2 & 0 & 0
    \end{bmatrix}.
\end{align}
Employing MATLAB symbolic tools, we can check that 
\begin{align}
 \| \J_2 \|_\rho \leq 0.966
\end{align}
holds for our choice of parameters. Furthermore, the result of 
Lemma~\ref{lemma:spectral_radius_mu} guarantees
\begin{align}
    \|\J_1\|_\rho\leq 1-10^{-1/2}\sqrt{\mu}.
\end{align}
Moreover, we calculate $\boldsymbol{\xi}$ as
\begin{align}
    \boldsymbol{\xi} & =\bsigma\bsigma^\top-\diag(\bsigma)^2 =\begin{bmatrix} \mu^2 & \mu \\ \mu & 1 \end{bmatrix} - \begin{bmatrix} \mu^2 & 0 \\ 0 & 1 \end{bmatrix}=\begin{bmatrix} 0 & \mu \\ \mu & 0 \end{bmatrix}
\end{align}
whose nor is bounded as
\begin{align}
    \|\boldsymbol{\xi}\|\leq\mu.
\end{align}
$\mu\leq0.02$ concludes the proof:
\begin{align}
    \|\C\|_{\rho_\epsilon} & \leq \max\{1-10^{-1/2}\sqrt{\mu},0.966\}+\epsilon+0.27\mu \\
    & \leq 1-\sqrt{\mu}/4+\epsilon.
\end{align}
Choosing $\epsilon=0.05\sqrt{\mu}$ concludes the proof. 
\end{proof}

\begin{lemma}[Restated Lemma \ref{lemma:uniform_gaurantees}] \label{lemma:uniform_convergence_app}
For $\kappa\leq0.02$, consider running stochastic acceleration method on example~\eqref{exam:uniform_example}. If $\alpha=2$, $\beta=1-10^{-1/2}\sqrt{\kappa}$ and $\gamma=\kappa/10$, then
\begin{align}
    \E \| \w_n - \w_* \|^2 \leq \frac{1200}{\sqrt{\kappa}} \left( 1- \sqrt{\kappa}/5 \right)^n \| \w_0 - \w_* \|^2.
\end{align}
holds.
\end{lemma} 

\begin{proof}
The first coordinate of $\x$ is a Rademacher random variable with fourth moment $k_1=1/4$ and $\sigma_1=1/2$. For the second coordinate, the moments of uniform distribution on range $[-\kappa^{-1/2},\kappa^{-1/2}]$ are
\begin{align}
    k_2=\int_{-\kappa^{-1/2}}^{\kappa^{-1/2}}\frac{x^4}{2\kappa^{-1/2}}dx=\kappa^{-2}/5
\end{align}
and
\begin{align}
    \sigma_2=\int_{-\kappa^{-1/2}}^{\kappa^{-1/2}}\frac{x^2}{2\kappa^{-1/2}}dx=\kappa^{-1}/3.
\end{align}
Plug $k_i$'s and $\sigma_i$'s into $\J_i$ as defined in Lemma \ref{lemma:spectral_radius_bound}
\begin{align}
    \J_{1} := \begin{bmatrix} 
      [\D_1]_{11}^2   &  2[\D_1]_{11}  & 1 \\
    [\D_1]_{11}[\D_2]_{11}   & [\D_{2}]_{11} & 0 \\ 
    [\D_{2}]_{11}^2  & 0 & 0
    \end{bmatrix} 
\end{align}
and
\begin{align}
    \J_{2} := \begin{bmatrix} 
      [\D_1]_{22}^2 + 4(1+\alpha)^2 \gamma^2 \kappa^{-2}/45   &  2[\D_1]_{22}  & 1 \\
    [\D_1]_{22}[\D_2]_{22} -  4\alpha (1+\alpha)  \gamma^2  \kappa^{-2}/45   & [\D_{2}]_{22} & 0 \\ 
    [\D_{2}]_{2}^2 + 4\gamma^2 \alpha^2  \kappa^{-2}/45  & 0 & 0
    \end{bmatrix}. 
\end{align}

For the Rademacher coordinate, the additional noise term cancels out. Clearly $-[\D_2]_{11}$ is an eigenvalue and a simple deduction from the proof for Lemma \ref{lemma:spectral_radius_mu} yields $\|\J_1\|_\rho \leq 1-10^{-1/2}\sqrt{\kappa}$. For $\J_2$, we employ MATLAB to verify following fact.
\paragraph{Fact 3.}
\textit{
For $\kappa\leq0.02,\alpha=2,\beta=1-10^{-1/2}\sqrt{\kappa}$ and $\gamma=\kappa/10$, $\|\J_2\|_\rho\leq0.965.$}

Next we calculate $\boldsymbol{\xi}$,
\begin{align}
    \boldsymbol{\xi} & =\bsigma\bsigma^\top-\diag(\bsigma)^2 \\
    & =\begin{bmatrix} 1 & \kappa^{-1}/3 \\ \kappa^{-1}/3 & \kappa^{-2}/9 \end{bmatrix} - \begin{bmatrix} 1 & 0 \\ 0 & \kappa^{-2}/9 \end{bmatrix}=\begin{bmatrix} 0 & \kappa^{-1}/3 \\ \kappa^{-1}/3 & 0 \end{bmatrix}
\end{align}
whose norm is 
$\|\boldsymbol{\xi}\| = \kappa^{-1}/3$.
And $\|\C\|_{\rho_\epsilon}$ is upper bounded as
\begin{align}
    \|\C\|_{\rho_\epsilon} & \leq \max\{1-10^{-1/2}\sqrt{\kappa},0.95\}+\epsilon+0.09\kappa \\
    & \leq 1-\sqrt{\kappa}/4+\epsilon.
\end{align}
Choosing $\epsilon = 0.05\sqrt{\kappa}$ concludes the proof. 
\end{proof}

\subsubsection{Spectral radius at $\mu$}\label{sec:spectral_radius_mu}
 Here we establish an upper bound for the spectral radius of $3\times3$ matrix $\J_1$ when $\sigma_1=\mu$. Subscript $i$ are temporarily omitted without ambiguity.

\paragraph{Root of characteristic equation.}
In the rest part of this subsection, we will fix $\sigma=\mu$ and set $\alpha,\gamma$ and $\beta$ as defined in Lemma \ref{lemma:gaussian_gaurantees}, on which all claims and lemmas mentioned below are based without special statements. We denote the entries of $\J$ as
\begin{align}
    \J = \begin{bmatrix} 
      D_1^2 + 2(1+\alpha)^2 \gamma^2 \mu^2   &  2 D_1  & 1 \\
    D_1D_2 -  2\alpha (1+\alpha)  \gamma^2 \mu^2  & D_{2} & 0 \\ 
    D_{2} + 2\gamma^2 \alpha^2 \mu^2 & 0 & 0
    \end{bmatrix}.
\end{align}
where
\begin{align}
    D_1=(1+\beta)-(1+\alpha)\gamma\mu,\quad D_2=\alpha\gamma\mu-\beta.
\end{align}
Now, consider the characteristic equation of $\J$
\begin{align}\label{eq:J_charact}
    x^3+bx^2+cx+d=0
\end{align}
where 
\begin{align}
    b &{} :=-D_1^2 + 2(1+\alpha)^2 \gamma^2 \mu^2-D_{2} \\
    c & := 2(1+\alpha)^2 \gamma^2 \mu^2D_{2} - D_{2} - 2\gamma^2 \alpha^2 \mu^2 - D_1^2D_2 +  4\alpha (1+\alpha)  \gamma^2 \mu^2 D_1 \\
    d & :=D_{2}\big(D_{2} + 2\gamma^2 \alpha^2 \mu^2\big).
\end{align}
 The following fact about the roots of Eq. (\ref{eq:J_charact}) can be verified by MATLAB symbolic tools.
\paragraph{Fact 1.}
\textit{For $\mu\leq0.1$, the discriminant of Eq. (\ref{eq:J_charact}), i.e. $\Delta=b^2c^2-4c^3-4b^3d-27d^2+18bcd$, is positive. Therefore the characteristic equation of $\J$ has one real root $x_1$ and two conjugate complex roots $x_2,x_3$.}

We turn our focus into the real root $x_1$. In Lemma \ref{lemma:spectral_radius_mu}, we extend our bounds to absolut values of complex roots. Due to the complexity of closed-form roots of the above cubic equation, we  approximate them. 
\paragraph{Main idea.}
The key idea is based on viewing $\J$ as a perturbation of matrix $\J_s$ which reads as
\begin{align}
    \J_s = \begin{bmatrix} 
      D_1^2   &  2 D_1  & 1 \\
    D_1D_2  & D_{2} & 0 \\ 
    D_{2}  & 0 & 0
    \end{bmatrix}.
\end{align}
For small $\gamma$, we expect that $\J_s-\J$ be close to zero. The characteristic equation of $\J_s$ is
\begin{align}\label{eq:Js_charact}
    y^3+b_sy^2+c_sy+d_s=0
\end{align}
where 
\begin{align}
    b  :=-D_1^2 -D_{2}, \quad  c  := - D_{2} - D_1^2D_2, \quad d :=D_{2}^2 .
\end{align}

Notice that $y=-D_2$ is a root of Eq. (\ref{eq:Js_charact}). These simple solutions provide us proper estimates of the real root of Eq. (\ref{eq:J_charact}).

Let $z=x-y$. Plugging $x=y+z$ and $y=-D_2$ into Eq. (\ref{eq:J_charact}) and Eq. (\ref{eq:Js_charact}) respectively and subtracting from both sides leads to a cubic equation about $z$
\begin{align}\label{eq:z_cubic}
    z^3+pz^2+qz+r=0
\end{align}
where 
\begin{align}
    p := b-3D_2, \quad q :=-2bD_2+3D_2^2+c, r :=(b-b_s)D_2^2-(c-c_s)D_2+d-d_s.
\end{align}
Let  $z_1,z_2$ and $z_3$ be roots of above cubic equation. These roots relate to those of Eq. (\ref{eq:J_charact}) as
$
 x_i =-D_2+z_1
$
for $i=1,2,3$. A natural consequence of this is $z_1$ is real and $z_2$, $z_3$ are conjugate complex (considering that $D_2$ is real). 
\paragraph{Properties of the real root.}
Let's focus on the real root $z_1$. MATLAB verification indicates coefficients $p$, $q$ and $r$ obey the following property.

\paragraph{Fact 2.}
\textit{
For $\mu\leq0.1$, coefficients in Eq. (\ref{eq:z_cubic}) satisfy $p,\ q> 0$ and $r<0$.
}  

Given the above fact, we prove that the real root $z_1$ is positive. 
\begin{lemma}
For $\mu\leq0.1$, the real root $z_1$ of Eq.  (\ref{eq:z_cubic}) is positive.
\end{lemma}

\begin{proof}
Consider three roots $z_1,\ z_2$ and $z_3$ of Eq. (\ref{eq:z_cubic}). We calculate
\begin{align}
    (z-z_1)(z-z_2)(z-z_3) & =z^3-(z_1+z_2+z_3)z^2+(z_1z_2+z_2z_3+z_3z_1)z-z_1z_2z_3 \\
    & =z^3+pz^2+qz+r
\end{align}
which boils down to $z_1z_2z_3=-r$. The existence of zero roots is ruled out since $r\ne0$. The last fact implies that
\begin{align}
    z_1=\frac{z_1z_2z_3}{|z_2|^2}=-\frac{r}{|z_2|^2}>0
\end{align}
since $z_2$ and $z_3$ are conjugate complex roots.
\end{proof}

\paragraph{Cubic root approximation.} 
Now, we establish an upper-bound on $z_1$.
\begin{lemma}
For $\mu\leq0.1$, $0\leq z_1\leq-r/q$ holds.
\end{lemma}
\begin{proof}
The last fact implies $p>0,q>0,r>0$ and $z_1>0$. By rearranging of terms in Eq. (\ref{eq:z_cubic}), we have 
\begin{align}\label{eq:z_quad}
    pz^2+q^z=-(r+z_1^3)
\end{align}
of which $z_1$ is still a root. Since $z_1>0$ then $d+z_1^3=-bz_1^2-cz_1<0$ holds.
\begin{align}
    z_1\overset{\textrm{(A)}}{\leq} \frac{-q+\sqrt{q^2-4p(d+z_1^3)}}{2p}\overset{\textrm{(B)}}{<}\frac{-q+\sqrt{q^2-4pd}}{2p}\leq-\frac{r}{q}
\end{align}
where (A) is due to $0<-4p(d+z_1^3)<-4pd$ and (B) comes from inequality $\sqrt{1+x}\leq1+x/2$ for $x\geq 0$.
\end{proof}

Since both $q$ and $r$ are polynomials in $\sigma=\mu$ and other parameters, a rational-form upper bound on $z_1$ can be established. This allows establishing a bound on $x_1=z_1-D_2$ by analyzing $m$ and $n$.

\begin{lemma}\label{lemma:upper_bound_x}
  For $\mu\leq0.1$, $-D_2\leq x_1\leq1-10^{-1/2}\sqrt{\mu}$ holds.
\end{lemma}
\begin{proof}
The lower bound is immediate. Let's consider the compact notation $\theta=1-\beta$. Our choice of parameters leads to $\theta^2=0.1\mu\leq0.01$. Then, we write $q,r$ as polynomials in $\theta$
\begin{align}
    q & = 54\theta^6 + 51\theta^5 - 23\theta^4 - 9\theta^3 + 3\theta^2\geq2\theta^4>0
\end{align}
and 
\begin{align}
    r & = - 24\theta^7 - 4\theta^6 + 16\theta^5 - 4\theta^4\geq- 4\theta^4<0.
\end{align}
Now we calculate
\begin{align}
    z_1\leq-r/q=\frac{4\theta^4}{2\theta^2}\leq2\theta^2.
\end{align}
This leads to an upper bound on $x_1$
\begin{align}
    x_1=-D_2+z_1\leq 1-\theta-2\theta^2+2\theta^2\leq1-\theta=1-10^{-1/2}\sqrt{\mu}.
\end{align}
\end{proof}

\paragraph{Bound for spectral radius.}
So far, we have proven the real eigenvalue is bounded by $1-\mathcal{O}(\sqrt{\mu})$ It remains to bound complex eigenvalues.
\begin{lemma} \label{lemma:spectral_radius_mu}
For $\mu\leq0.1$, $\alpha=2$, $\beta=1-10^{-1/2}\sqrt{\mu}$ and $\gamma=0.1$, $\|\M\|_\rho\leq1-10^{-1/2}\sqrt{\mu}$.
\end{lemma}
\begin{proof}
We reuse the notation $\theta=1-\beta$. The determinant of $\J$ read as 
\begin{align}
    \det\J & =1-3\theta-3\theta^2+11\theta^3+14\theta^4-20\theta^5-24\theta^6
\end{align}
is upper bounded by $|D_2|(1-\theta)^2$, which can be verified by MATLAB or manual calculation. 
On the other hand, we have $x_2x_3=|x_2|^2>0$ since $x_2$ and $x_3$ are conjugate. Consider $x_1>-D_2>0$ and calculate the spectral radius 
\begin{align}
    \|\J\|_\rho & =\max\{x_1,|x_2|\}\max\{x_1,\sqrt{\frac{\det\J}{x_1}}\} \\
    &\leq\max\{x_1,\sqrt{\frac{|D_2|(1-\theta)^2}{|D_2|}}\}\overset{\textrm{(A)}}{\leq}1-\theta=1-10^{-1/2}\sqrt{\mu}
\end{align}
where (A) comes from $x_1\leq1-\theta$ as shown in Lemma \ref{lemma:upper_bound_x}.
\end{proof}

\end{document}